\pdfoutput=1
\newif\ifpersonal
\personaltrue 
\documentclass[12pt,a4paper,dvipsnames,x11names]{amsart} 
\usepackage{amsmath,amsthm,amssymb,mathrsfs,mathtools,bm,eucal,tensor} 
\usepackage{stmaryrd}
\usepackage[all,cmtip]{xy}
\usepackage{dsfont}
\usepackage{microtype,lmodern} 
\usepackage[utf8]{inputenc} 
\usepackage[T1]{fontenc} 
\usepackage{enumerate,comment,braket,xspace,tikz-cd,csquotes} 
\usepackage[centering,vscale=0.75,hscale=0.75]{geometry}

\usepackage[hidelinks,colorlinks=true,linkcolor=Red4,citecolor=RoyalBlue]{hyperref}
\usepackage[capitalize]{cleveref}
\usepackage{mathpazo}
\usepackage{euler}

\newtheorem{thm}[subsection]{Theorem}
\newtheorem{theorem}[subsection]{Theorem}
\newtheorem{prop}[subsection]{Proposition}

\newtheorem{cor}[subsection]{Corollary}
\newtheorem{lem}[subsection]{Lemma}
\newtheorem{lemma}[subsection]{Lemma}

\theoremstyle{definition}

\newtheorem{definition}[subsection]{Definition}

\newtheorem{recollection}[subsection]{Recollection}

\newtheorem{rem}[subsection]{Remark}

\newtheorem{example}[subsection]{Example}
\newtheorem{construction}[subsection]{Construction}

\numberwithin{equation}{subsection}
\newtheorem{notation}[subsection]{Notation}

\usepackage{etoolbox}
\patchcmd{\section}{\scshape}{\bfseries}{}{}
\makeatletter
\renewcommand{\@secnumfont}{\bfseries}
\makeatother

\newcommand{\sep}{\mathrm{sep}}

\newcommand{\cl}{\mathrm{cl}}
\newcommand{\fm}{\mathfrak m}

\newcommand{\sw}{\mathrm{sw}}

\newcommand{\red}{\mathrm{red}}

\newcommand{\sO}{\mathscr{O}}

\newcommand{\cK}{\mathcal K}
\newcommand{\cZ}{\mathcal Z}
\newcommand{\cF}{\mathcal F}
\newcommand{\bZ}{\mathbb Z}
\newcommand{\sF}{\mathcal{F}}

\newcommand{\bA}{\mathbb A}

\newcommand{\bQ}{\mathbb Q}
\newcommand{\bT}{\mathbb T}

\newcommand{\cE}{\mathcal E}

\newcommand{\cO}{\mathcal O}
\newcommand{\cC}{\mathcal C}
\newcommand{\cH}{\mathcal H}
\newcommand{\cL}{\mathcal L}
\newcommand{\cI}{\mathcal I}

\newcommand{\bN}{\mathbb N}
\newcommand{\Qlbar}{\overline{\bQ}_\ell}  

\DeclareMathOperator{\CoInd}{CoInd}

\DeclareMathOperator{\Sw}{Sw}

\DeclareMathOperator{\rk}{rk}

\DeclareMathOperator{\Ker}{Ker}

\DeclareMathOperator{\sbar}{\overline{s}}
\DeclareMathOperator{\Cons}{Cons}
\DeclareMathOperator{\lc}{lc}

\DeclareMathOperator{\Coh}{Coh}

\DeclareMathOperator{\length}{length}
\DeclareMathOperator{\tors}{tors}

\DeclareMathOperator{\Gal}{Gal}

\DeclareMathOperator{\Spec}{Spec}

\DeclareMathOperator{\Ima}{Im}
\DeclareMathOperator{\Weil}{Weil}

\DeclareMathOperator{\op}{op}
\DeclareMathOperator{\Fun}{Fun}
\DeclareMathOperator{\Loc}{Loc}
\DeclareMathOperator{\LC}{Loc}

\usepackage[english]{babel}

\usepackage[english]{babel}

\title{Bounding ramification with coherent sheaves}

\begin{document}

\author[H.Hu]{Haoyu Hu} 
\address{School of Mathematics, Nanjing University, Hankou Road 22, Nanjing, China}
\email{huhaoyu@nju.edu.cn, huhaoyu1987@gmail.com}

\author[J.-B. Teyssier]{Jean-Baptiste Teyssier}
\address{Institut de Math\'ematiques de Jussieu, 4 place Jussieu, Paris, France}
\email{jean-baptiste.teyssier@imj-prg.fr}

\begin{abstract}
Given a coherent sheaf $\cE$ on a scheme of finite type $X$ over a perfect field, we introduce a category of complexes of étale sheaves on $X$ with logarithmic conductors bounded by $\cE$ and study its compatibilities with finite push-forward.
\end{abstract}

\maketitle

\setcounter{tocdepth}{1}
\tableofcontents

\section{Introduction}

The goal of this paper is to use coherent sheaves to bound  the wild ramification of complexes of étale sheaves.
\medskip 

Let $X$ be a normal scheme of finite type over a perfect field $k$ of characteristic $p>0$, let $D\subset X$ be an effective Cartier divisor  and put $U:=X-D$.
Let $\Lambda$ be a finite local ring of residue characteristic $\ell \neq p$ and let $\cL$ be a locally constant constructible sheaf of free $\Lambda$-modules of finite type on $U$.
A useful way to bound the wild ramification of $\cL$ along $D$ is to look at the Swan conductors of the restrictions of 
$\cL$ to curves $C$ passing through the points of $D$.
Namely, we say following Deligne \cite{Esnault_Kerz_Deligne} that $\cL$ has Swan conductors bounded by $D$ if for every smooth connected curve $C$ over $k$ and every morphism $f : C\to X$ over $k$ with $f(C)\nsubseteq D$, we have for every $x\in C$,
\begin{equation}\label{sw_bounded}
\Sw_x \cL|_{C\cap U}\leq m_x(f^*D) 
\end{equation}
where $\Sw_x \cL|_{C\cap U}$ is the Swan conductor of  $\cL|_{C\cap U}$ at $x\in C$ in the sense of \cite{CL} and $m_x(f^*D)$ is the multiplicity of $f^*D$ at $x$.
When $X$ is smooth over $k$ and when $D$ has normal crossings, Abbes and Saito's ramification theory \cite{RamImperfect} implies that the poset of effective Cartier divisors supported on $D$ and satisfying \eqref{sw_bounded} has a minimal element, called the Swan divisor \cite{Barrientos,Hu_log_ram}.\medskip

The interest of this notion comes from the observation that objects with bounded Swan conductors tend to exhibit behaviors similar to their characteristic zero counterparts.
For example in the tame case, that is when $D=0$, a Lefschetz theorem for the tame fundamental group is known  \cite{DriDel,Esnault_Kindler_tame}.
When $k$ is finite, Hiranouchi proved without assumption on the singularities of $X$ that there are only a finite number of étale coverings on $X$ with bounded degree and wild ramification bounded by a given effective Cartier divisor at infinity \cite{Hiranouchi}, which is reminiscent to the finite generation of the fundamental group of a complex algebraic variety.
Also, the Betti numbers of locally constant constructible sheaves with bounded rank and wild ramification are bounded \cite{HuTeyssier_bound_etale}.
In characteristic zero, this phenomenon was observed for irregular flat bundles on surfaces in \cite{HuTeyssier_bound_diff}. \medskip

Note however that effective Cartier divisors are not sensitive to the wild ramification in codimension $\geq 2$ displayed by sheaves which are not locally constant.
This confines the characteristic $p>0$ side of the above analogies to locally constant sheaves on normal varieties, while the most natural set-up  is sometimes that of arbitrary complexes of constructible sheaves on arbitrary schemes of finite type.
To solve this problem, we extend the above definition by replacing effective Cartier divisors with arbitrary coherent sheaves.
For a scheme of finite type $X$ over $k$, let $\bQ[\Coh(X)]$ be the free $\bQ$-vector space on the set of isomorphism classes of coherent sheaves on $X$.
Given $\cE\in \bQ[\Coh(X)]$ and  $\Lambda$ a finite local ring of residue characteristic $\ell \neq p$, we introduce a full subcategory 
$$
D_{ctf}^b(X,\cE,\Lambda) \subset D_{ctf}^b(X,\Lambda)
$$ 
of tor finite complexes with logarithmic conductors bounded by $\cE$.
In a nutshell, the right hand side of \eqref{sw_bounded} gets replaced by the length of the torsion part of $f^*\cE$ at $x$ and the Swan conductor gets replaced by the highest ramification slope of $\cL|_{C\cap U}$ at $x$, also called the \textit{logarithmic conductor} of $\cL|_{C\cap U}$ at $x$. 
See \cref{bounded_conductor} for a precise definition.\medskip

The main result of this paper is the following compatibility of $D_{ctf}^b(X,\cE,\Lambda)$ with finite push-forward :

\begin{thm}[\cref{finite_direct_image}]\label{thm_1}
Let $f : Y\to X$ be a finite morphism between  schemes of finite type over $k$.
Let $\Sigma$ be a stratification of $Y$ and let $\cE\in \bQ[\Coh(Y)]$.
Then there exists $\cE'\in \bQ[\Coh(X)]$ such that for every $\Sigma$-stratified constructible complex $\cK\in D_{ctf}^b(Y,\cE,\Lambda)$, we have $f_{\ast}\cK \in D^b_{ctf}(X,\cE',\Lambda)$.
\end{thm}

The proof of \cref{thm_1} reduces to the analysis of the compatibility of Abbes and Saito's logarithmic conductor \cite{RamImperfect} with finite push-forward.
To state its main output,  recall that for a normal scheme of finite type $X$ and for a constructible sheaf of $\Lambda$-modules $\cF$  of finite tor-dimension on $X$, we let $LC_X(\cF)$ be the Weil divisor of $X$ whose multiplicity at $Z$ is the generic logarithmic conductor of $\cF$ along $Z$, in the sense of Abbes and Saito (see \cref{conductor_divisor_def}).
Then we have the following 
\begin{thm}[\cref{finite_direct_image_lcc_case}]\label{thm_2}
Let $f:Y\to X$ be a finite morphism of normal schemes of finite type over $k$. 
Let $D$ be a reduced effective Cartier divisor on $X$ and put $U:=X-D$.
Define $E:=f^{-1}(D)$ and put $j : V:=Y-E\hookrightarrow Y$.
For every $\cL\in \Loc_{tf}(V,\Lambda)$, we have
$$
LC_X(f_{*}j_!\cL)\leq LC_X(f_{*}j_!\Lambda)+ f_* LC_Y(j_!\cL) \ .
$$ 
\end{thm}
Note that \cref{thm_2} also holds for the non logarithmic conductor divisor (see \cref{finite_direct_image_lcc_case}). \medskip

\cref{thm_2} is one of the crucial ingredients to prove the Betti estimates from \cite{HuTeyssier_bound_etale}.
Its proof is local.
In a nutshell, it rests upon Hu's semi-continuity results  \cite{Hu_Leal} for reducing the proof  to a local statement for curves treated in the unpublished note \cite{Teyladic}.
For a variant of \cref{thm_2}  for the Swan conductor obtained by a global argument via the Grothendieck-Ogg-Shafarevich formula, see \cite{Esnault_srinivas_bounding_ram}.

\subsection*{Acknowledgement}
We thank A. Abbes, T. Saito, Y. Taguchi for their interest.
We are grateful to the anonymous referee for a careful reading and for suggesting numerous improvements.
H. H. is supported by the National
Natural Science Foundation of China (Grants No. 12471012) and the Natural Science Foundation of Jiangsu Province (Grant No. BK20231539).

\begin{notation}\label{general_notation}
We introduce the following running notations.
\begin{enumerate}\itemsep=0.2cm
\item[$\bullet$] $k$  denotes a perfect field of characteristic $p>0$.

\item[$\bullet$] The letter $\Lambda$ will refer to a finite local ring of residue characteristic $\ell \neq p$.

\item[$\bullet$] For a scheme $X$ of finite type over $k$, we denote by $D_{ctf}^b(X,\Lambda)$ the derived category of complexes of $\Lambda$-sheaves of finite tor-dimension with bounded and constructible cohomology sheaves.

\item[$\bullet$]
We let $\Cons_{ctf}(X,\Lambda)$ be the category of constructible sheaves of $\Lambda$-modules of finite tor-dimension over $X$ and
$\LC_{tf}(X,\Lambda)\subset \Cons_{tf}(X,\Lambda)$ the full subcategory spanned by  locally constant constructible sheaves.
By \cite[Lemma 4.4.14]{Weibel_intro}, the germs of any $\cL\in \LC_{tf}(X,\Lambda)$ are automatically free  $\Lambda$-modules of finite rank.
 
\item[$\bullet$]
For a finite stratification $\Sigma$ of $X$, we let $D_{\Sigma,tf}^b(X,\Lambda)\subset D_{tf}^b(X,\Lambda)$ be the full subcategory spanned by $\Sigma$-constructible complexes, that is complexes $\cK\in D_{tf}^b(X,\Lambda)$ such that for every $Z\in \Sigma$, the cohomology sheaves of $\cK|_Z$ lie in $\LC_{tf}(Z,\Lambda)$.

\end{enumerate}

\end{notation}
\section{Conductor and finite direct image: local case}

\begin{notation}\label{local_fields_notation}

Let $K$ be a henselian discrete valuation field, $\sO_K$ the ring of integer of $K$, $\fm_K$ the maximal ideal of $\sO_K$, $F$ the residue field of $\sO_K$, $\overline{K}$ an algebraic closure of $K$ and $K^{\sep}\subset \overline{K}$ the separable closure of $K$ in $\overline{K}$.
Let $G_K$ be the Galois group of $K^{\sep}$ over $K$.
We denote by  $I_K\subset G_K$ the inertia subgroup and by  $P_K\subset I_K$ the wild ramification subgroup.
\end{notation}

\begin{recollection}\label{local_fields}

 In \cite{RamImperfect},  Abbes and Saito defined two decreasing filtrations $\{G^r_K\}_{r\in\bQ_{>0}}$ and $\{G^r_{K,\log}\}_{r\in\bQ_{\geq 0}}$ on $G_K$ by closed normal subgroups.   
They are called respectively the {\it ramification filtration} and the {\it logarithmic ramification filtration}.  
For $r\in\bQ_{\geq 0}$, put
\begin{equation*}
G^{r+}_K=\overline{\bigcup_{s>r}G_K^s}\ \ \ \textrm{and}\ \ \ G^{r+}_{K,\log}=\overline{\bigcup_{s>r}G_{K,\log}^s}.
\end{equation*}

\end{recollection}

\begin{prop}[{\cite{RamImperfect,asii,logcc,wr}}]\label{propramfil}
The following inclusions hold:
\begin{enumerate}
\item
For any $0<r\leq 1$, we have 
$$
G^r_K=G^0_{K,\log}=I_K \text{ and } G^{1+}_K=G^{0+}_{K,\log}=P_K.
$$
\item
For any $r\in \bQ_{\geq 0}$, we have 
$$
G^{r+1}_K\subseteq G^r_{K,\log}\subseteq G^r_K.
$$
If $F$ is perfect, then for any $r\in \bQ_{\geq 0}$, we have 
$$
G^r_{K,\cl}=G^r_{K,\log}=G^{r+1}_K.
$$ 
where $G^r_{K,\cl}$ is the classical wild ramification subgroup as defined in \cite{CL}.
\end{enumerate}
\end{prop}

Let $M$ be a finitely generated $\Lambda$-module with continuous $P_K$-action. 
 The module $M$ has decompositions 
 \begin{equation}\label{twodecomp}
M=\bigoplus_{r\geq 1}M^{(r)}\ \ \ \textrm{and}\ \ \ M=\bigoplus_{r\geq 0}M_{\log}^{(r)}
\end{equation}
into $P_K$-stable submodules, where $M^{(1)}=M^{(0)}_{\log}=M^{P_K}$, and such that for every $r\in \bQ_{>0}$,
\begin{align*}
(M^{(r+1)})^{G^{r+1}_K}=0\ \ \ \textrm{and}\ \ \ (M^{(r+1)})^{G^{(r+1)+}_K}=M^{(r+1)};\\
(M^{(r)}_{\log})^{G^{r}_{K,\log}}=0\ \ \ \textrm{and}\ \ \ (M^{(r)}_{\log})^{G^{r+}_{K,\log}}=M^{r}_{\log}.
\end{align*}
The decompositions \eqref{twodecomp} are called respectively the {\it slope decomposition} and the {\it logarithmic slope decomposition} of $M$. 
The values $r$ for which $M^{(r)}\neq 0$ (resp. $M^{(r)}_{\log}\neq 0$) are the {\it slopes} (resp. the {\it logarithmic slopes}) of $M$.\medskip

We will make repeated use of the following facts :

\begin{lem}[{\cite[Proposition 1.1,Lemma 1.5]{Katz_Gauss_sum}}]\label{slope_exact}
For every $r\in \bQ_{\geq 0}$,  the endofunctors of the category of finitely generated $\Lambda$-modules with continuous $P_K$-action defined by 
$$
M\to M^{(r)}     \quad \text{ and }  \quad  M\to M_{\log}^{(r)}
$$
are exact.
If furthermore $M$ is free of finite type over $\Lambda$, so is $M_{\log}^{(r)}$.
\end{lem}

\begin{lem}[{\cite[Lemma 1.5]{Katz_Gauss_sum}}]\label{slope_tensor}
Let $\Lambda\to \Lambda'$ be a morphism between finite local rings with residue field of characteristic $\ell \neq p$.
For every $r\in \bQ_{\geq 0}$ and every finitely generated free $\Lambda$-module $M$  with continuous $P_K$-action,  $r$ is a slope (resp.  logarithmic slope) of $M$ if and only if $r$ is a slope (resp.  logarithmic slope) of $\Lambda'\otimes_{\Lambda} M$.
\end{lem}

\begin{lem}[{\cite[Lemma 1.8]{Katz_Gauss_sum}}]\label{GKstable}
In the setting of \cref{local_fields}, assume that $M$ is a finitely generated $\Lambda$-module with continuous $G_K$-action. 
Then, each factor $M^{(r)}$ and $M_{\log}^{(r)}$ from \eqref{twodecomp} is stable under $G_K$.
\end{lem}

\begin{definition}\label{notation_logconductor}
Let $M$ be a finitely generated $\Lambda$-module with continuous $P_K$-action. 
We denote by $c_K(M)$ the largest slope of $M$ and refer to $c_K(M)$ as the \textit{conductor of $M$}.
Similarly, we denote by $\lc_K(M)$ the largest logarithmic slope of $M$ and refer to $\lc_K(M)$ as the \textit{logarithmic conductor of $M$}.
We say that $M$ is {\it isoclinic} (resp. {\it logarithmic isoclinic}) if $M$ has only one slope (resp. only one logarithmic slope). 
\end{definition}
The following is an immediate consequence of  \cref{propramfil}-(2).

\begin{lem}\label{inequalityLogNonLog}
We have 
$$
\lc_{K}(M)  \leq c_K(M)\leq  \lc_{K}(M)+1 \ . 
$$
If the residue field $F$ is perfect, we have
$$
\lc_{K}(M)+1 =c_K(M) \ .  
$$
\end{lem}

\begin{lem}\label{slopes_from_coho_to_complex}
Let $M^{\bullet}$ be a  complex of finitely generated $\Lambda$-modules with continuous $G_K$-action and let $i\in \bZ$.
If $r\in \mathds{Q}_{\geq 0}$ is a logarithmic slope of $\cH^i M^{\bullet}$, then $r$  is a logarithmic slope of $M^i$.
\end{lem}

\begin{proof}
Let $Z^i := \Ker(M^i\to M^{i+1})\subset M^i$.
By \cref{slope_exact}, it is enough to show that  $r$  is a logarithmic slope of $Z^i$.
By \cref{slope_exact} again , the surjective map $Z^i\to  \cH^i M^{\bullet}$ induces a surjective map 
$$
(Z^i)^{(r)}_{\log} \to (\cH^i M^{\bullet})^{(r)}_{\log} \ .
$$
Since the right hand side is non zero, so is the left hand side and the conclusion follows.
\end{proof}

\begin{lem}\label{good_presentation_complex}
Let $M^{\bullet}$ be a bounded finite tor-dimension complex of  
finitely generated  $\Lambda$-modules with continuous $G_K$-action.
Then, $M^{\bullet}$ is quasi-isomorphic to a bounded complex $N^{\bullet}$ of free $\Lambda$-modules of finite type with continuous $G_K$-action such that
$r\in \mathds{Q}_{\geq 0}$ is a logarithmic slope of some $N^i$ if and only if it is a logarithmic slope of some $\cH^i N^{\bullet}$.
\end{lem}

\begin{proof}
By \cite[03TT]{SP}, we can assume that $M^{\bullet}$ is a bounded complex of free $\Lambda$-modules of finite type with continuous $G_K$-action.
We have 
$$
M^{\bullet} = \bigoplus_{r\in \mathds{Q}_{\geq 0}} M^{\bullet, (r)}_{\log} 
$$
where each factor is a bounded complex of free $\Lambda$-modules of finite type by \cref{slope_exact}.
By \cref{GKstable}, the continuous $G_K$-action on $M^{\bullet}$  restricts to a continuous $G_K$-action on each $M^{\bullet, (r)}_{\log}$.
If $r\in \mathds{Q}_{\geq 0}$ is not a logarithmic slope of any $\cH^i M^{\bullet}$, then by \cref{slope_exact}, we have
$$
\cH^i (M^{\bullet, (r)}_{\log})\simeq (\cH^i M^{\bullet})^{(r)}_{\log}\simeq 0 \ . 
$$
Thus $M^{\bullet, (r)}_{\log}$ is acyclic. Hence, at the cost of removing the direct factors $M^{\bullet, (r)}_{\log}$ from $M^{\bullet}$ for all $r$ as above, we can assume that if $r\in \mathds{Q}_{\geq 0}$ is a logarithmic slope of some $M^i$, then it is a logarithmic slope of some $\cH^i M^{\bullet}$.
Since the converse is  true by \cref{slopes_from_coho_to_complex}, 
the proof of \cref{good_presentation_complex} is complete.
\end{proof}

\begin{cor}\label{slope_base_change_module}
Let $M^{\bullet}$ be a bounded finite tor-dimension complex of  
finitely generated  $\Lambda$-modules with continuous $G_K$-action.
Let $\Lambda\to \Lambda'$ be a morphism between finite local rings of residue characteristic $\ell \neq p$.
If $r\in \mathds{Q}_{\geq 0}$ is a logarithmic slope of some $\cH^i (M^{\bullet}\otimes_{\Lambda}^{L} \Lambda')$,  then it is a logarithmic slope of some $\cH^i M^{\bullet}$.
\end{cor}

\begin{proof}
By \cref{good_presentation_complex}, we can assume that $M^{\bullet}$ is a bounded complex of free $\Lambda$-modules of finite type with continuous $G_K$-action such that $r\in \mathds{Q}_{\geq 0}$ is a logarithmic slope of some $M^i$ if and only if it is a logarithmic slope of some $\cH^i M^{\bullet}$.
Let $i\in \bZ$ and let $r\in \mathds{Q}_{\geq 0}$ be a logarithmic slope of $\cH^i (M^{\bullet}\otimes_{\Lambda}^{L} \Lambda')$.
Since, 
$$
M^{\bullet}\otimes_{\Lambda}^{L} \Lambda'  \simeq M^{\bullet}\otimes_{\Lambda} \Lambda'   \ ,
$$
we deduce from \cref{slopes_from_coho_to_complex} that $r$ is a logarithmic slope of $M^i\otimes_{\Lambda} \Lambda'$.
By \cref{slope_tensor}, we deduce that $r$ is a logarithmic slope of $M^i$. 
By construction, we get that $r$ is a logarithmic slope of some $\cH^j M^{\bullet}$.
\end{proof}

\subsection{Co-induction and ramification}

\begin{recollection}
Let $G$ be a profinite group and let $H\subset G$ be a closed subgroup. 
Following \cite[\textsection 2.5]{gal},  the scalar restriction from  $\Lambda$-modules with continuous $G$-action to $\Lambda$-modules  with continuous $H$-action admits  a right adjoint called the \textit{co-induction}. 
By construction, the co-induction assigns to a $\Lambda$-module $M$ with continuous $H$-action the $\Lambda$-module  with continuous $G$-action described as
$$
\CoInd^H_G(M) :=\{f : G\to M \text{ continuous commuting with the left action of $H$}\}
$$
where $M$ is endowed with  the discrete topology.
By construction $g\in G$ acts on $\CoInd^H_G(M)$ by 
$$
(g\cdot f)(x)= f(xg) \ .
$$

\begin{lem}\label{invariant_codinduction}
Let  $H''\subset H'\subset H$ be closed subgroups such that $H'$ and $H''$ are normal in $G$.
Then, the following holds:
\begin{enumerate}\itemsep=0.2cm
\item $\CoInd^H_G(M)^{H'}=\{f\in \CoInd^H_G(M) \text{ such that } f(G)\subset M^{H'}  \}$.

\item If furthermore  $H$ is open in $G$,  we have 
$$
\CoInd^H_G(M)^{H''}=\CoInd^H_G(M)^{H'} \text{ if and only if } M^{H''}=M^{H'} \ .
$$
\end{enumerate}

\end{lem}
\begin{proof}
We prove (1).
Let $f\in \CoInd^H_G(M)^{H'}$.
For $h\in H'$ and $g\in G$, we have $g^{-1}hg \in H'$.
Thus
$$
(g^{-1}hg\cdot f)(g)= f(g)=f(gg^{-1}hg)=h\cdot f(g)
$$
Hence,  $f(G)\subset M^{H'}$.
The reverse inclusion is proved similarly.
We now prove (2) and assume that $H$ is open in $G$.
The converse implication is tautological due to (1).
Let us prove the direct implication.
Since $H''\subset H'$, we have $M^{H'}\subset M^{H''}$.
Let $x\in M^{H''}$ and consider a set of representatives $g_1,\dots, g_n$ for $H\backslash G$.
Let $f : G\to M$ be the function defined by
$$
f(hg_i)=h\cdot x 
$$
for every $h\in H$ and $i=1,\dots, n$.
Since $H\subset G$ is open,  the continuity of $f$ can be checked on  each $H g_i$.
On the other hand,  the restriction of $f$ to $H g_i$ decomposes as
$$
\begin{tikzcd} 
 H g_i   \ar{r}{\cdot g_i^{-1}}    &  H   \ar{r}{\cdot x } &  M 
\end{tikzcd} 
$$
which is indeed continuous as a composition of continuous maps.
Thus,  $f \in \CoInd^H_G(M)$.
Since $H''$ is normal in $H$ and $x\in M^{H''}$,  we have $f(G)\subset M^{H''}$.
By (1), we deduce that $f$ lies in $\CoInd^H_G(M)^{H''}=\CoInd^H_G(M)^{H'}$.
By (1) again, we deduce $f(g_1)=x\in M^{H'}$.
This concludes the proof of \cref{invariant_codinduction}.
\end{proof}

\end{recollection}

\begin{lem}\label{conductor_extension}
In the setting of  \cref{local_fields}, let $K'/K$ be a finite separable extension in $K^{\sep}$.
Let $ r\in \mathds{Q}_{> 0}$.
Then, the following are equivalent: 
\begin{enumerate}\itemsep=0.2cm
\item $G_{K,\log}^{r} \subset G_{K'}$.
\item $G_{K,\log}^{r}$ acts trivially on $\CoInd^{G_{K'}}_{G_K}(\Lambda)$.
\item $r>\lc_K(\CoInd^{G_{K'}}_{G_K}(\Lambda))$.
\end{enumerate}
\end{lem}

\begin{proof}
The equivalence between (2) and (3) is obvious.
Suppose that (2) holds.
Since $ G_{K'}$ is open in  $G_{K}$, we have 
$$
\CoInd^{G_{K'}}_{G_K}(\Lambda) \simeq \Fun(G_{K'}\backslash G_K,\Lambda) \ .
$$
Let $g$ be an element of $G_K$ acting trivially on $\CoInd^{G_{K'}}_{G_K}(\Lambda)$.
Let $f : G_{K'}\backslash G_K\to \Lambda$ sending $G_{K'}$ to 1 and sending a class distinct from $G_{K'}$  to 0.
Then,  $g\cdot f = f$ yields $g\in G_{K'}$.
Thus,  (1) is true.
Suppose that (1) holds and let us prove (2).
It is enough to prove that $G_{K,\log}^{r}$ acts trivially on $G_{K'}\backslash G_K$, which follows immediately from (1) and the fact that $G_{K,\log}^{r}$ is normal in $G_K$.
\end{proof}

\begin{definition} 
In the setting of \cref{local_fields}, let $L/K$ be a finite extension in $\overline{K}$ and let $K'/K$ be the separable closure of $K$ in $L$.  
We define
$$
\lc_{L/K} :=\lc_K(\CoInd^{G_{K'}}_{G_K}(\Lambda))\in \mathds{Q}_{\geq 0}   \ .
$$
\end{definition}

\begin{lemma}\label{equality_ramification_group}
In the setting of  \cref{local_fields},  assume that the residue field $F$ is perfect.
Let $K'/K$ be a finite  extension  in $K^{\sep}$.
Then for every $r > \lc_{K'/K}$, we have 
$$
G_{K,\log}^{r}= G_{K',\log}^{\psi_{K'/K}(r)}  
$$
 in $G_{K'}$, where $\psi_{K'/K} : \mathds{R}_{\geq 0}\to \mathds{R}_{\geq 0}$ is Herbrand's function \cite[IV \textsection 3]{CL}.
\end{lemma}
\begin{proof}
The argument below is extracted from the unpublished note \cite{Teyladic}.
Let $L/K$ be a finite Galois extension in $K^{\sep}$ containing $K'$.
In particular, the extension $L/K'$ is finite Galois.
Put  $G=\Gal(L/K)$ and $H=\Gal(L/K')\subset G$.
Since $r > \lc_{K'/K}$, we have  $G_{K,\log}^{r} \subset G_{K'}$.
Hence,   $G_{K,\log}^r \subset H$.
Thus,
$$
G_{K,\log}^{r}  =  H\cap G_{K,\log}^{r}  =   H\cap G_{\psi_{L/K}(r)} 
   =  H_{\psi_{L/K}(r)} = H_{K',\log}^{ \varphi_{L/K'}\circ \psi_{L/K}(r)}
$$
where the third equality comes from the compatibility of the lower-numbering ramification filtration with subgroups \cite[IV Proposition 2]{CL}.
From \cite[IV Remark 2]{CL}, we have 
$\psi_{L/K}= \psi_{L/K'}\circ \psi_{K'/K}$.
Hence, 
$$
\varphi_{L/K'}\circ\psi_{L/K}=  \psi_{K'/K}  \ .
$$
Thus, 
\begin{equation}\label{Gr_and_H}
\Gal(L/K)_{K,\log}^{r} = \Gal(L/K')_{K',\log}^{\psi_{K'/K}(r)} \ .
\end{equation}
Let $\mathcal{P}$ be  the poset  of finite Galois extensions $L/K$ in $K^{\sep}$ containing $K'$, ordered by the inclusion.
Then, $\mathcal{P}$ is cofinal both in the poset of finite Galois extensions of $K$ in $K^{\sep}$ and in the poset of finite Galois extensions of $K'$ in $K^{\sep}$.
\cref{equality_ramification_group} thus follows from (\ref{Gr_and_H}) by passing to the limit over $\mathcal{P}^{\op}$.
\end{proof}

\begin{lem}[{\cite{Teyladic,KatoH_wild_ram_res_curves}}]\label{equality_ramification_group_ins}
In the setting of  \cref{local_fields},  assume that the residue field $F$ is perfect.
Let $L/K$ be a finite purely inseparable extension  in $\overline{K}$ and let $ L^{\sep}$ be the separable closure of $L$ in $\overline{K}$.
Then,  the  isomorphism $\iota : G_{L}\xrightarrow{\sim} G_{K}$ induced by restriction from $L^{\sep}$ to $K^{\sep}$ is compatible with the logarithmic ramification filtration, that is 
$$
\iota(G_{L,\log}^r)=G_{K,\log}^r 
$$
for every $r\in \mathds{Q}_{\geq 0}$.
\end{lem}

\begin{cor}\label{equality_ramification_group_general}
In the setting of  \cref{local_fields},  assume that the residue field $F$ is perfect.
Let $L/K$ be a finite  extension  in $\overline{K}$ and let $K'/K$ be the separable closure of $K$ in $L$.
Then for every $r > \lc_{L/K}$, we have 
$$
G_{K,\log}^r= \iota(G_{L,\log}^{\psi_{K'/K}(r)}  )
$$
where $\iota : G_{L}\xrightarrow{\sim}G_{K'}$ is the isomorphism induced by restriction from $L^{\sep}$ to $K^{\sep}$  in $\overline{K}$.
\end{cor}
\begin{proof}
Since $r > \lc_{L/K}= \lc_{K'/K}$,  we have 
$$
G_{K,\log}^{r}    = G_{K',\log}^{\psi_{K'/K}(r)} = \iota(G_{L,\log}^{\psi_{K'/K}(r)}  ) 
$$
where the first equality follows from \cref{equality_ramification_group} and the second one from 
\cref{equality_ramification_group_ins}.

\end{proof}

\begin{prop}\label{slope_induction}
In the setting of  \cref{local_fields},  assume that the residue field $F$ is perfect.
Let $L/K$ be a finite  extension  in $\overline{K}$ and let $K'/K$ be the separable closure of $K$ in $L$.
Let $M$ be a finitely generated $\Lambda$-module with continuous $G_{L}$-action.
For $r>\lc_{L/K}$, the following are equivalent : 
\begin{enumerate}\itemsep=0.2cm
\item $r$  is a logarithmic slope of $\CoInd^{G_{L}}_{G_{K}}(M) $.
\item $\psi_{K'/K}(r)$  is a logarithmic slope of $M$.
\end{enumerate}
\end{prop}
\begin{proof}
Note that $L/K'$ is purely inseparable.
In the above statement, the coinduction is made along the morphism $G_L \xrightarrow{\sim} G_{K'} \subset G_K$.
Since $r>\lc_{L/K}=\lc_{K'/K}$, we have $G_{K,\log}^r \subset G_{K'}$ in virtue of \cref{conductor_extension}  and the groups 
$$
G_{K,\log}^{r+}\subset G_{K,\log}^{r}  \subset G_{K'} 
$$
satisfy the assumptions of \cref{invariant_codinduction} with $G=G_K$ and $H=G_{K'}$ open in $G_K$.
Thus, 
\begin{align*}
&  r \text{ is not a logarithmic slope of } \CoInd^{G_{L}}_{G_{K}}(M)   &    \\
\Longleftrightarrow  &\CoInd^{G_{L}}_{G_{K}}(M)^{G_{K,\log}^{r} }=\CoInd^{G_{L}}_{G_{K}}(M)^{G_{K,\log}^{r+}}       &  \\   
   \Longleftrightarrow &    M^{G_{K,\log}^{r} }=M^{G_{K,\log}^{r+}}          & \text{ \cref{invariant_codinduction}}\\ 
 \Longleftrightarrow  &   M^{G_{L,\log}^{\psi_{K'/K}(r) } }=M^{G_{L,\log}^{\psi_{K'/K}(r) +}}   &   \text{ \cref{equality_ramification_group_general}  } \\
\Longleftrightarrow  &  \psi_{K'/K}(r) \text{ is not a logarithmic slope of }  M.
\end{align*}
\end{proof}

The following lemma is proved in  \cite[IV Proposition 12]{CL} for finite \textit{Galois} extensions.
We however needs the case of an arbitrary finite separable extension.

\begin{lemma}\label{phi_inequality}
In the setting of  \cref{local_fields},  assume that the residue field $F$ is perfect.
Let $K'/K$ be a finite separable extension  in $K^{\sep}$.
Then, the following   hold :
\begin{enumerate}\itemsep=0.2cm
\item The function $\varphi_{K'/K}$ is strictly increasing, piecewise linear and concave.
\item For every $r\geq 0$, we have $\varphi_{K'/K}(r)\leq r$.
\end{enumerate}
\end{lemma}
\begin{proof}
Let $L/K$ be the Galois closure of $K'$ in $K^{\sep}$.
Put $G:=\Gal(L/K)$ and $H:=\Gal(L/K')\subset G$.
By definition, we have 
$$
\varphi_{K'/K}= \varphi_{L/K}\circ  \psi_{L/K'} \ . 
$$ 
Hence, $\varphi_{K'/K}$ is strictly  increasing and piecewise linear as composition of strictly increasing and piecewise linear functions.
To prove that $\varphi_{K'/K}$ is concave, it is thus enough to prove that the restriction of $\varphi_{K'/K}$ to its affine locus has decreasing derivative.
For $r\geq 0$ sufficiently generic, we have 
\begin{align*}
\varphi_{K'/K}'(r)&  = \psi_{L/K'}'(r) \varphi_{L/K}'(\psi_{L/K'}(r)) & \\
                          & =\frac{|H_0|}{|H_{\psi_{L/K'}(r)}|}\frac{|G_{\psi_{L/K'}(r)}|}{|G_0|} & \text{ \cite[IV Proposition 12]{CL}  }\\
                          & =\frac{|\Ima(G_{\psi_{L/K'}(r)}\to G/H)|}{|\Ima(G_{0}\to G/H)|}   &                  
\end{align*}
where the last equality follows from $H_u=H\cap G_u$ for every $u\geq 0$.
Since the ramification filtration  is decreasing, we deduce that so is $\varphi_{K'/K}'$.
This proves (1).
We now prove (2). 
By loc. cit., we have   $\psi_{L/K'}(0)= \varphi_{L/K}(0)=0$, so that $\varphi_{K'/K}(0)=0$.
Since $\varphi_{K'/K}$  is piecewise linear  and concave,  we are thus left to show that the right derivative of $\varphi_{K'/K}$ at 0 is smaller than 1.
This follows immediately from the above formula for $\varphi_{K'/K}'(r)$ with $r$ generic.
\end{proof}

\begin{cor}\label{slope_induction_final}
In the setting of  \cref{local_fields},  assume that the residue field $F$ is perfect.
Let $L/K$ be a finite  extension  in $\overline{K}$.
Let $M$ be a finitely generated $\Lambda$-module with continuous $G_{L}$-action.
Then,
$$
\lc_K(\CoInd^{G_{L}}_{G_{K}}(M) )\leq \max(\lc_{L/K},\lc_L(M))   \ .
$$
\end{cor}
\begin{proof}
Put $r= \lc_K(\CoInd^{G_{L}}_{G_{K}}(M) )$.
If $r> \lc_{L/K},$  \cref{slope_induction}  implies that $\psi_{K'/K}(r)$  is a logarithmic slope of $M$.
In particular $\psi_{K'/K}(r)\leq \lc_L(M)$.
Thus,
$$
r\leq \varphi_{K'/K}(\lc_L(M)) \leq  \lc_L(M)
$$
where the second inequality follows from  \cref{phi_inequality}-(2).
\end{proof}

\section{Conductor divisors}\label{semi_continuity_conductors}\label{semi_continuity_section}
Let $X$ be a normal scheme of finite type over $k$.
Let $Z$ be an integral Weil divisor and let $\eta\in Z$ be its generic point.
Let $K$ be the fraction field of $\hat{\mathcal{O}}_{X,\eta}$ and fix a separable closure $K^{\sep}$ of $K$.
For  $\cF\in \Cons_{tf}(X,\Lambda)$, the pull-back $\cF|_{\Spec K}$ is a $\Lambda$-module of finite type with continuous $G_{K}$-action.
Using the notations from \cref{notation_logconductor}, we put
$$
c_{Z}(\cF):= c_{K}(\cF|_{\Spec K})  \text{ and }
 \lc_{Z}(\cF):= \lc_{K}(\cF|_{\Spec K})  \ .
$$

\begin{definition}\label{conductor_divisor_def}
Let $X$ be a normal scheme of finite type over $k$ and let $\cF\in \Cons_{tf}(X,\Lambda)$.
We define the \textit{conductor divisor  of $\cF$} as the Weil divisor with rational coefficients given by
$$
C_X(\cF):=\sum_{Z}  c_{Z}(\cF) \cdot  Z
$$
and the \textit{logarithmic conductor divisor  of  $\cF$} as the Weil divisor with rational coefficients given by
$$
LC_X(\cF):=\sum_{Z} \lc_{Z}(\cF) \cdot  Z
$$
where the sums run over the set of integral Weil divisors of $X$.
\end{definition}

\begin{rem}
When there is no ambiguity, we will drop the subscript  $X$ in $C_X(\cF)$ and $LC_X(\cF)$.
\end{rem}

\begin{rem}\label{inequalityLogNonLog_divisor_rem} \cref{inequalityLogNonLog} implies 
$$
LC_X(\cF)\leq C_X(\cF) \leq LC_X(\cF) + D \ .
$$
where $D$ is the support of $C_X(\sF)$.
\end{rem}

\begin{definition}
In the setting of \cref{semi_continuity_section}, we define the \textit{generic conductor} and the \textit{generic (logarithmic) conductor} of $\cF$ along a divisor $D$ respectively by
$$
c_D(\cF):= \max_{Z} c_Z(\cF) \text{ and }
 \lc_D(\cF):= \max_{Z}  \lc_Z(\cF)   
$$
where $Z$ runs over the set of irreducible components of $D$.
\end{definition}

 The conductor and log conductor divisors enjoy the following semi-continuity property :

\begin{theorem}[{\cite[Theorem 1.4,1.5]{Hu_Leal}}]\label{semi_continuity}
Let $f : Y\to X$ be a morphism of smooth schemes of finite type over $k$.
Let $D$ be an effective Cartier divisor on $X$ and put $U:=X- D$.
Assume that $E:=Y\times_X D$ is an effective  Cartier divisor on $Y$ and that $\Lambda$ is a finite field. 
For every  $\cL\in \Loc_{tf}(U,\Lambda)$, we have
$$
C_Y( (j_!\cL)|_Y) \leq f^* C_X( j_!\cL) \ .    
$$
If furthermore $D$ has normal crossings,  we have
$$
 LC_Y((j_!\cL)|_Y) \leq f^* LC_X(j_!\cL)     \ .
$$
\end{theorem}

The conductor divisor can be detected by curves, due to the following 

\begin{prop}[{\cite[Proposition 2.22]{wr},\cite[Proposition 3.3]{HuTeyssier_Leal}}]\label{equality_DT}
Let $X$ be a smooth scheme over $k$ of pure dimension $n$. 
Let $D$ be an effective Cartier divisor on $X$ and put $j : U:=X- D\hookrightarrow X$. 
Let $\cL\in \LC_{tf}(U,\Lambda)$.
Then, there is a closed conical subset $SS \subset \bT^*X$ of pure dimension $n$  and a dense open subset $\Omega \subset D$ such that for every immersion $i:S\hookrightarrow X$  over $k$ where $S$ is a smooth curve satisfying
\begin{enumerate}\itemsep=0.2cm
\item  $S$ meets $D$ transversely at a single smooth point $x\in \Omega$,
\item every non zero $\omega \in SS_x$ does not vanish on $\bT_xS\subset  \bT_x X$,
\end{enumerate}  
we have $C_{S}( (j_!\cL)|_{S}) = i^* C_X(j_!\cL)$.
\end{prop}

\begin{theorem}[{\cite[Theorem 1.5]{Hu_Leal}}]\label{Hu_generic_swan_via_curve}
Let $X$ be a smooth scheme over $k$. 
Let $D$ be a smooth irreducible divisor on $X$ and put $j : U:=X- D\hookrightarrow X$. 
For every $\cL\in \Loc_{tf}(U,\Lambda)$, we have
$$
\lc_D(j_! \cL) = \sup_{\cI(X,D)} \frac{\lc_x(j_!\cL)}{m_x(f^*D)}
$$
where $\cI(X,D)$ is the set of triples $(S,f : S\to X,x)$ where $f : S\to X$ is an immersion from a smooth curve over $k$ to $X$ such that $x= S\cap D$ is a closed point of $X$.
\end{theorem}

%


\begin{prop}[{\cite[Corollary 5.8]{HuTeyssier_semi_continuity}}]\label{inequality_LC}
Let $f : X\to S$ be a smooth morphism between smooth schemes of finite type over $k$.
Let $D\subset X$ be an effective Cartier divisor  relative to $S$ such that $f|_D : D\to S$ is smooth.
Put $j : U := X-D\hookrightarrow X$.
Then for every $\cL\in \Loc_{tf}(U,\Lambda)$ and every algebraic geometric point $\sbar\to S$, we have
$$
C((j_!\cL)|_{X_{\sbar}}) \leq i_{\sbar}^* C(j_!\cL)  \text{ and }  LC((j_!\cL)|_{X_{\sbar}}) \leq i_{\sbar}^*  LC(j_!\cL) 
$$
where $i_{\sbar} : X_{\sbar} \to X$ is the canonical morphism.
\end{prop}

\begin{rem}
\cref{semi_continuity},\cref{equality_DT}, \cref{Hu_generic_swan_via_curve} and \cref{inequality_LC} are proved for a finite field $\Lambda$ of characteristic $\ell\neq p$ in \cite{wr,Hu_Leal,HuTeyssier_semi_continuity}.
They hold for $\Lambda$  a finite local ring of residue characteristic $\ell \neq p$ as a consequence of \cref{slope_tensor}.
\end{rem}

\section{Torsion divisors of coherent sheaves}

\begin{construction}
Let $X$ be a normal noetherian scheme and let $\cE\in \Coh(X)$.
If $X^1\subset X $ denotes the set of codimension 1 points of $X$, we define a Weil divisor on $X$ by the formula
$$
T(\cE):= \sum_{\eta\in X^1}    \length_{\cO_{X,\eta}}( \cE|_{X_{\eta}}^{\tors})\cdot \overline{\{\eta\}}
$$
where $X_{\eta}= \Spec \cO_{X,\eta}$ and where $\cE|_{X_{\eta}}^{\tors}$  is the torsion part of $\cE|_{X_{\eta}}$.
We refer to $T(\cE)$ as the \textit{torsion divisor of $\cE$}. \\ \indent
We denote by 
$$
\bQ[\Coh(X)]
$$ 
the free $\bQ$-vector space on the set of isomorphism classes of coherent sheaves on $X$. \\ \indent
We denote by 
$$
\bQ_{\geq 0}[\Coh(X)]\subset \bQ[\Coh(X)]
$$
the sub-monoid formed by linear combinations with coefficients in $\bQ_{\geq 0}$.

If $\Weil(X)_{\bQ}$ is the space of $\bQ$-Weil divisors on $X$, the map $T : \Coh(X)\to \Weil(X)_{\bQ}$ induces a map of $\bQ$-vector spaces
$$
T : \bQ[\Coh(X)]\to \Weil(X)_{\bQ} \ .
$$
Furthermore, if $f : Y\to X$ is a morphism between normal noetherian schemes, the (non derived) pullback 
$f^* : \Coh(X) \to \Coh(Y)$ induces a map of $\bQ$-vector spaces
$$
f^* :\bQ[\Coh(X)] \to \bQ[\Coh(Y)] \ .
$$
\end{construction}

\begin{example}
If $D$ is an effective Cartier divisor of $X$, then $T(\cO_D)=D$.
\end{example}

\begin{lem}\label{RE_Cartier}
Let $X$ be a normal scheme and let $Z,Z'\subset X$ be effective Cartier divisors.
Then
$$
T(\cO_{Z+Z'})= T(\cO_Z) + T(\cO_{Z'})\ .
$$
\end{lem}
\begin{proof}
By definition, we have $\cI_{Z+Z'}= \cI_Z\cdot \cI_{Z'}$.
The conclusion thus follows by localizing at a codimension 1 point of $X$.
\end{proof}

\begin{lem}\label{devissage_E_0}
Let $X$ be a normal noetherian scheme and let 
$$
0\to \cE'\to \cE \to \cE''\to 0
$$
be an exact sequence in $\Coh(X)$.
Then, we have 
$$
T(\cE)\leq T(\cE')+ T(\cE'') \ .
$$
If furthermore $\cE$ is a torsion sheaf, we have 
$$
T(\cE)=T(\cE')+ T(\cE'') \ .
$$
\end{lem}

\begin{proof}
The first claim follows from the fact that localization is exact and that the torsion functor is left-exact.
The last claim follows from the additivity of the  length function.
\end{proof}

\begin{recollection}\label{normal_morphism}
Let us recall that a morphism $f : X\to S$ between schemes of finite type over $k$ is \textit{normal} if it is flat with geometrically normal fibres. 
By \cite[Proposition 11.3.13]{EGA4-4}, the source of a normal morphism is normal if the target is normal.
By \cite[Proposition 2.3.4]{EGA4-2},  a generic point of $X$ is sent to a generic point of S.
In particular, if $X$ is irreducible the pull-back of $X$ along any dense open immersion $V\hookrightarrow S$ is irreducible as well.
\end{recollection}

\begin{notation}
For a scheme $X$ of finite type over $k$, a $\bQ$-Weil divisor $D$ of $X$ and an irreducible component $Z$ of $D$, we denote by $m_Z(D)\in \bQ$ the multiplicity of $D$ along $Z$. 
We let $m(D)$ be the maximal multiplicity of $D$ if $D\neq 0$ and put $m(D)=0$ if $D=0$.
\end{notation}

\begin{construction}\label{c_X/S}
Let $f: X\to S$ be a normal morphism between schemes of finite type over $k$.
For  $\cE\in \Coh(X)$,  we define a function $\chi_{\cE}\colon S\to\bN$ by
$$
\chi_{\cE}:S\to\bN,\ \ \ s\mapsto m(T(\cE|_{X_{\overline s}}))
$$
where $\overline s\to S$ is an algebraic geometric point above $s\in S$.
We put
$$
\mu_{f}(\cE):= \sup (\{0\} \cup \chi_{\cE}(S))\in \bN\cup \{\infty\} \ .
$$ 

\end{construction}

The following proposition is not used in the rest of the paper :

\begin{prop}\label{dominant_by_constructible_max_multi}
Let $f: X\to S$ be a normal morphism between schemes of finite type over $k$.
For every $\cE\in \Coh(X)$, the function $\chi_{\cE}\colon S\to\bN$  is constructible.
\end{prop}

\cref{dominant_by_constructible_max_multi} admits the following immediate 

\begin{cor}\label{bound_cRE}
Let $f: X\to S$ be a normal morphism between schemes of finite type over $k$.
For every $\cE\in \Coh(X)$,  the quantity $\mu_{f}(\cE)$ is finite.
\end{cor}

\begin{lem}\label{bound_chi_reduction}
Let $f: X\to S$ be a normal morphism between integral schemes of finite type over $k$.
Let $Z$ be a closed subscheme of $X$ of codimension at least 2.
Then, $\chi_{\cE}$  and $\chi_{\cE|_{X-Z}}$  coincide on a dense open subset of $S$.
\end{lem}
\begin{proof}
We can assume that $Z$ is reduced.
Let $Z_1,\dots, Z_n$ be its irreducible components.
By \cite[Proposition 11.4.1]{Vakil_rising_sea} applied to $f : X\to S$ and to the $Z_i\to S$,  at the cost of shrinking $S$, we can suppose that for every $s\in S$, the fibre $X_s$ has pure dimension $\dim X-\dim S$ and the $Z_{i,s}$ have pure dimension $\dim Z_i-\dim S$.
Hence, for every $s\in S$ we have 
$$
\dim Z_{i,s} = \dim Z_i-\dim S= \dim X_s + \dim Z_i- \dim X \ .
$$
Since $\dim Z_i <\dim X-1$, we deduce
$$
\dim Z_{i,s} <\dim X_s -1 
$$
where $X_s$ and $Z_{i,s}$ have pure dimension.
Thus, $Z_{s}$ has codimension at least 2 in $X_s$.  
Hence, we have $\chi_{\cE}(s)=\chi_{\cE|_{X-Z}}(s)$.
\end{proof}

\begin{proof}[Proof of \cref{dominant_by_constructible_max_multi}]
Let $\cE\in \Coh(X)$.
We run the  following d\'evissage :
\begin{itemize}\itemsep=0.2cm
\item[(i)] Since  $S$ is noetherian, 
it is enough to find a dense open subset $V\subset S$  such that $ \chi_{\mathcal E}|_V=\chi_{\mathcal E|_{f^{-1}(V)}}$ is constructible.

\item[(ii)]
At the cost of replacing $S$ by $S^{\mathrm{red}}$ we may assume that $S$ is reduced.

\item[(iii)]  By (i) and (ii), we are reduced to the case where $S=\mathrm{Spec}(A)$ is affine, connected and smooth over $k$.

\item[(iv)] Let $\{U_i\}_{1\leq i\leq m}$ be a finite Zariski open cover of $X$. 
Observe that 
$$
\chi_{\cE}=\max_{1\leq i\leq m}\chi_{\cE|_{U_i}} \ .
$$
Hence, we are left to show that each $\chi_{\cE|_{U_i}}$ is constructible.
Since $X$ is noetherian, we can thus assume that $X$ is affine.
By  \cref{normal_morphism},  we deduce that $X$ is normal affine.
By \cite[0357]{SP}, the scheme $X$ is a finite disjoint union of normal affine integral schemes.
Hence, at the cost of replacing $X$ by its irreducible components, we can assume that $X$ is integral normal affine.
By  \cref{normal_morphism},  its pullback stays so after replacing $S$ by an affine dense open subset.

\item[(v)]
Consider the exact sequence
$$
0 \to \cE_{\tors}\to \cE\to \cE'\to 0
$$
where $\cE_{\tors}$ is the torsion subsheaf of $\cE$ and where $\cE'$ is torsion free.
Since $X$ is normal, the sheaf $\cE'$ is locally free in a neighbourhood of every codimension 1 point of $X$.
Hence, $\cE'$ is locally free on the complement of a closed subset $Z\subset X$ of a codimension at least 2.
Applying  \cref{bound_chi_reduction} to $Z$ and (iv)  to $X-Z$, we can assume that $\cE'$ is locally free with $X$ is integral normal affine over $k$. 
In that case, the above sequence split.
Hence, we have 
$$
\chi_{\cE} = \chi_{\cE_{\tors}} + \chi_{\cE'} = \chi_{\cE_{\tors}}\ .
$$
Thus, we can further assume that $\cE$ is a torsion sheaf.

\item[(vi)]
By \cite[00L0]{SP},  the sheaf $\cE$ admits a finite filtration by coherent subsheaves
$$
0=\cE_0 \subset \cE_1  \subset \cdots \subset \cE_{n-1} \subset \cE_n=\cE
$$
such that for every $i=0,\dots, n-1$, we have $\cE_{i+1}/\cE_i \simeq\cO_{Z_i}$ for some closed irreducible subset $Z_i\subset X$.
By generic flatness theorem \cite[Théorème 6.9.1]{EGA4-2}, we may assume that each $Z_i$ is flat over $S$.
Since $\cE$ is torsion, so are each $\cE_i$.
Hence, the $\cO_{Z_i}$ are torsion.
In particular, $Z_i\subset X$ is a strict closed irreducible subset.
By \cref{bound_chi_reduction}, we can thus assume that each $Z_i$ is an effective Cartier divisor of $X$ relative to $S$.

\item[(vii)]
Let $Z\subset X$ be a closed subscheme.
By \cite[0550]{SP} applied to $Z\to S$, there is a pullback square
$$
	\begin{tikzcd}
		     Z'  \arrow{r}  \arrow{d}    &   Z    \arrow{d} \\
	     X'  \arrow{r}{h}  \arrow{d}    &   X    \arrow{d}{f} \\
	       S'            \arrow{r}{g}    & S  
	\end{tikzcd}
$$
where $g : S' \to S$ is open quasi-finite with $S'$ integral affine, where $Z'^{\red}$ is flat over $S'$ with geometrically reduced generic fibre.
Observe that
$$
\chi_{h^*\cE}(s') = \chi_{\cE}(g(s')) 
$$
for every $s'\in S'$.
At the cost of replacing $S$ by a dense open subset of $S'$, we are left to treat the case where $S$ is integral smooth affine over $k$, where $X$ is  integral normal affine and
$\cE$ admits a finite filtration by coherent subsheaves
$$
0=\cE_0 \subset \cE_1  \subset \cdots \subset \cE_{n-1} \subset \cE_n=\cE
$$
such that for every $i=0,\dots, n-1$, we have $\cE_{i+1}/\cE_i \simeq\cO_{Z_i}$ for some effective Cartier divisor $Z_i\subset X$ relative to $S$ where $Z_i^{\red}$ has geometrically reduced generic fibre over $S$.
By \cite[Theorem 9.7.7]{EGA4-3} at the cost of shrinking $S$,  we can further assume that the $Z_i^{\red}$ have geometrically reduced fibres over $S$.
For every algebraic geometric point $\overline s\to S$, the chain 
$$
0=\cE_0|_{X_{\overline{s}}} \subset \cE_1|_{X_{\overline{s}}}  \subset \cdots \subset \cE_{n-1}|_{X_{\overline{s}}} \subset \cE_n|_{X_{\overline{s}}}=\cE|_{X_{\overline{s}}}
$$
is a filtration of $\cE|_{X_{\overline{s}}}$ with torsion quotients.
Since $X_{\overline{s}}$ is integral, each $\cE_i|_{X_{\overline{s}}}$ is a torsion sheaf.
Thus, \cref{devissage_E_0} yields
$$
T(\cE|_{X_{\overline{s}}})= T(\cO_{Z_{1,\overline{s}}})+ \cdots +  T(\cO_{Z_{n,\overline{s}} }) \ .
$$
Let $\cZ$ be the set of integral effective Cartier divisors of $X$ relative to $S$ with geometrically reduced fibres over $S$.
In particular, we can write
$$
Z_i= \sum_{Z\in \cZ} e_{i,Z}\cdot Z  
$$
where $e_{i,Z} \in \bN$.
For every $s\in S$, \cref{RE_Cartier} gives
$$
T(\cO_{Z_{i,\overline{s}}})= \sum_{Z\in \cZ}  e_{i,Z}\cdot T(\cO_{Z_{\overline{s}}}) \ .
$$
Thus,
$$
T(\cE|_{X_{\overline{s}}})= \sum_{Z\in \cZ}  \left(\sum_{i=1}^n   e_{i,Z}\right)\cdot T(\cO_{Z_{\overline{s}}}) \ .
$$
Since the $Z_{\overline{s}}$ are \textit{reduced} effective Cartier divisors of $X_{\overline{s}}$, the multiplicities of the $T(\cO_{Z_{\overline{s}}})$ are equal to 1.
Thus, 
$$
T(\cE|_{X_{\overline{s}}})= \sum_{Z\in \cZ}  \left(\sum_{i=1}^n   e_{i,Z}\right)\cdot Z_{\overline{s}} \ .
$$
On the other hand for $Z,Z'\in \cZ$ distinct contributing to the $Z_i$, the subscheme $Z\times_X Z'\subset X$ has codimension at least 2.
By \cref{bound_chi_reduction}, at the cost of shrinking $S$, we can assume that for every $Z,Z'\in \cZ$ distinct contributing to the $Z_i$, the divisors $Z$ and $Z'$ are disjoint.
In particular,  for every $s\in S$ so are the divisors $Z_{\overline{s}}$ and  $Z'_{\overline{s}}$.
Hence,  for every $s\in S$ we have
$$
\chi_{\cE}(s)= \max_{Z\in \cZ}\{\sum_{i=1}^n   e_{i,Z}\}   \ .
$$
\end{itemize}
\end{proof}

\begin{definition}\label{pulback_weil_divisor}
Let $f : Y\to X$ be a morphism between normal noetherian schemes. 
We say that  \textit{$D\in \Weil(X)_{\bQ}$ pull-backs along $f : Y\to X$} if for every
integral closed subscheme $Z\subset X$ contributing to $D$,  the pull-back $Y\times_X Z$ has pure codimension 1 in $Y$.
\end{definition}

For every $D\in \Weil(X)_{\bQ}$ which pull-backs along $f : Y\to X$, there is a well defined pull-back $f^*D\in \Weil(Y)_{\bQ}$.
If every $D\in \Weil(X)_{\bQ}$ pull-backs along $f : Y\to X$, there is a well-defined $\bQ$-linear map 
$$
f^* : \Weil(X)_{\bQ}\to \Weil(Y)_{\bQ} \ .
$$
We describe two cases of interest.

\begin{example}\label{scalar_exte}
Let $X$ be a noetherian geometrically normal scheme over a field $k$ and let $K/k$ be a field extension. 
Then, every $D\in \Weil(X)_{\bQ}$ pull-backs along  $X_K \to X$.
\end{example}

\begin{example}\label{flat_pull}
Let $f : Y\to X$ be a dominant morphism between regular noetherian schemes. 
Then, every Weil divisor on $X$ is Cartier and similarly with $Y$. 
On the other hand, Cartier divisors on $X$ pull-back to $Y$ since $f : Y\to X$ is dominant.
Hence, every $D\in \Weil(X)_{\bQ}$ pull-backs along $f : Y\to X$.
\end{example}

\begin{lem}\label{inequality_mR}
Let $f : Y\to X$ be a morphism between normal noetherian schemes and let $\cE\in \Coh(X)$ such that $T(\cE)$ pull-backs along $f : Y\to X$.
Let $\xi  \in Y$ be a codimension 1 point such that $\eta = f(\xi) \in X$ is a codimension 1 point.
Then,  the multiplicities of $f^*T(\cE)$ and $T(f^*\cE)$ at $\xi$ are the same.
\end{lem}
\begin{proof}
Let $\pi$ be a uniformizer of $\cO_{X,\eta}$ and write
$$
\cE|_{X_{\eta}} \simeq \cO_{X_{\eta}}^r \oplus \cO_{X_{\eta}}/(\pi)^{ n_1}\oplus \cdots  \oplus  \cO_{X_{\eta}}/(\pi)^{n_k} \ .
$$
Then,  the multiplicity of  $T(\cE)$ at $\eta$ is $n_1+\cdots + n_k$.
If  $e$ is the valuation of $f^*\pi$ at $\xi$, the multiplicity  $f^* T(\cE)$ at $\xi$ is thus $(n_1+\cdots + n_k)\cdot e$.
If $\pi'$ is a uniformizer of $\cO_{Y,\xi}$ we have 
$$
(f^*\cE)|_{Y_{\xi}} \simeq \cO_{Y_{\xi}}^r \oplus \cO_{Y_{\xi}} /(\pi^{\prime})^{ e\cdot n_1}\oplus \cdots  \oplus  \cO_{Y_{\xi}} /(\pi^{\prime})^{e\cdot n_k} \ .
$$
Hence, the multiplicity  $T(f^*\cE)$ at $\xi$ is $(n_1+\cdots + n_k)\cdot e$.
\end{proof}

Observe that in the next lemma, no finiteness assumption is required.

\begin{lem}\label{flat_pullback}
Let $f : Y\to X$ be a flat morphism between normal noetherian schemes and let $\cE\in \Coh(X)$ such that $T(\cE)$ pull-backs along $f : Y\to X$.
Then,  
$$
f^*T(\cE)=T(f^*\cE) \ .
$$
\end{lem}
\begin{proof}
Let $\xi \in Y$ be a codimension 1 point contributing to $f^*T(\cE)$ and put $\eta=f(\xi)$.
Since $f^* T(\cE)\to T(\cE)$ is flat, we deduce by \cite[Proposition 2.3.4]{EGA4-2} that $\eta$ is a generic point of $T(\cE)$, and thus a codimension 1 point of $X$.
By  \cref{inequality_mR}, we deduce that $f^*T(\cE)$ and $T(f^*\cE)$ have the same multiplicity at $\xi$.
\\ \indent
To conclude, let $\xi \in Y$ be a codimension 1 point contributing to $T(f^*\cE)$ and put $\eta=f(\xi)$.
By \cref{inequality_mR}, we are left to show that $\eta$ is a codimension 1 point of $X$.
Assume otherwise.
Since $f : Y\to X$ is flat, we deduce by \cite[Proposition 2.3.4]{EGA4-2}  that $\eta$ is a generic point of $X$.
Hence, $\cE$ is locally free in a neighbourhood of $\eta$.
Thus, $f^*\cE$ is locally free in a neighbourhood of $\xi$.
Hence $\xi$ does not contribute to $T(f^*\cE)$, contradiction.
\end{proof}

\begin{cor}\label{same_cRE}
Let $X$ be a normal scheme of finite type over $k$ and let $\cE\in \bQ[\Coh(X)]$.
Then, $T(\cE)_{\overline{k}}=T(\cE_{\overline{k}})$ where the subscript refers to the pullback along $ X_{\overline{k}}\to X$.
\end{cor}
\begin{proof}
The above statement  makes sense by \cref{scalar_exte}.
We can reduce to the case where $\cE\in \Coh(X)$. 
This case follows from \cref{flat_pullback}.
\end{proof}

\begin{cor}\label{flat_pullback_reg}
Let $f : Y\to X$ be a flat morphism between regular schemes. 
For every $\cE\in \bQ[\Coh(X)]$, we have
$$
f^*T(\cE)=T(f^*\cE) \ .
$$
\end{cor}
\begin{proof}
The above statement makes sense by  \cref{flat_pull}.
We can reduce to the case where $\cE\in \Coh(X)$. 
This cases follows from \cref{flat_pullback}.
\end{proof}

\section{Bounding log conductors with coherent sheaves}\label{bounded_ramification_section}

\begin{definition}\label{bounded_conductor}
Let $X$ be a scheme of finite type over $k$.
Let $\cK\in D_{ctf}^b(X,\Lambda)$ and  $\cE\in \bQ[\Coh(X)]$.
We say that \textit{$\cK$ has log conductors bounded by $\cE$} if for every morphism $f : C\to X$ over $k$ where $C$ is a smooth curve over  $k$, we have
$$
LC(\cH^i\cK|_C)\leq T(f^*\cE) 
$$
for every $i\in \mathds{Z}$.
We denote by $D_{ctf}^b(X,\cE,\Lambda)$ the full subcategory of $D_{ctf}^b(X,\Lambda)$ spanned by objects having log conductors bounded by $\cE$.
\end{definition}

\begin{rem}
Note that one could define a variant of $D_{ctf}^b(X,\cE,\Lambda)$ by using Swan conductors instead of logarithmic conductors.
Our justification for choosing the latter comes from the shape of the Betti estimates from \cite{HuTeyssier_bound_etale}, which are polynomial in the logarithmic conductor $\lc_D(\cL)$ and linear in the rank.
In particular, the inequality $\lc_D(\cL)\leq \sw_D \cL$ yields Betti estimates which are again polynomial in the Swan conductor and linear in the rank.
Using Swan conductors as ramification bounds from the start would  again lead to Betti estimates which are polynomial in the Swan conductor and linear in the rank. 
However when expressed in terms of logarithmic conductors, the inequality $\sw_D \cL \leq \lc_D(\cL)\cdot \rk \cL$ would yield a polynomial dependency in the rank.
This explains our choice for logarithmic conductors against Swan conductors in the definition of $D_{ctf}^b(X,\cE,\Lambda)$.
\end{rem}

\begin{rem}
If $Z = a_1\cdot Z_1 + \cdots +  a_n\cdot Z_n$ is a formal sum of closed subschemes with rational coefficients and if we put 
$$
\cE = \bigoplus_{i=1}^n a_i\cdot \cO_{Z_i} \in  \bQ[\Coh(X)] \ ,
$$
we will denote $D_{ctf}^b(X,Z,\Lambda)$ instead of $D_{ctf}^b(X,\cE,\Lambda)$ and say that $\cK\in D_{ctf}^b(X,\Lambda)$ has log conductors bounded by $Z$ when $\cK$ lies in $D_{ctf}^b(X,Z,\Lambda)$.
\end{rem}


\begin{lem}\label{slope_base_change_sheaf}
Let $X$ be a scheme of finite type over $k$, let $\cE\in \bQ[\Coh(X)]$ and let $\Lambda\to \Lambda'$ be a morphism between finite local rings of residue characteristic $\ell \neq p$.
For every $\cK\in D_{ctf}^b(X,\cE,\Lambda)$, we have $\cK\otimes_{\Lambda}^{L} \Lambda' \in D_{ctf}^b(X,\cE,\Lambda')$.
\end{lem}

\begin{proof}
By \cite[03TT]{SP} every $\cK\in D_{ctf}^b(X,\Lambda)$ can be represented by a bounded complex of constructible sheaves of flat $\Lambda$-modules.
In particular, the germ of $\cK$ at any point can be represented by a bounded finite tor-dimension complex of finitely generated $\Lambda$-modules.
 Then, \cref{slope_base_change_sheaf} follows from \cref{slope_base_change_module}.
\end{proof}

\begin{recollection}\label{adic_presentation}
Let $X$ be a scheme of finite type over $k$.
Then, for every $\cK\in D^b_c(X,\overline{\bQ}_\ell)$, there is a finite extension $L/\bQ_{\ell}$ and an integral representative $\cK_{\bullet}=(\cK_m)_{m\geq 0}$ for $\cK$.
If we put $\Lambda_m := \cO_L/\mathfrak{m}_L^m$, the sheaf $\cK_m$ is an object of $D^b_{ctf}(X,\Lambda_m)$ such that $\Lambda_m\otimes_{\Lambda_{m+1}}^{L} \cK_{m+1} \simeq \cK_{m}$.
\end{recollection}

The following definition upgrades \cref{bounded_conductor} to $\Qlbar$-coefficients.

\begin{definition}
Let $X$ be a scheme of finite type over $k$.
Let $\cK\in D_{c}^b(X,\Qlbar)$ and  $\cE\in \bQ[\Coh(X)]$.
We say that \textit{$\cK$ has log conductors bounded by $\cE$} if there is a finite extension $L/\bQ_{\ell}$ and an integral representative $\cK_{\bullet}=(\cK_m)_{m\geq 0}$ for $\cK$ such that for every $m\geq 0$, we have $\cK_m\in D^b_{cft}(X,\cE,\Lambda_m)$.
We denote by $D_c^b(X,\cE,\Qlbar)$ the full subcategory of $D_c^b(X,\Qlbar)$ spanned by objects having log conductors bounded by $\cE$.
\end{definition}

\begin{lem}\label{many_properties}
Let $X$ be a scheme of finite type over $k$.
Let  $\cE\in \bQ[\Coh(X)]$ and $\cK\in D_{ctf}^b(X,\Lambda)$.

\begin{enumerate}\itemsep=0.2cm

\item\label{pullback} Assume that $\cK\in D_{ctf}^b(X,\cE,\Lambda)$ and let $f : Y\to X$ be a morphism of  schemes of finite type over $k$ .
Then $f^*\cK\in D_{ctf}^b(Y,f^*\cE,\Lambda)$.

\item\label{direct sum} Assume that $\cK\in D_{ctf}^b(X,\cE,\Lambda)$.  
Then for every $\cE'\in \bQ_{\geq 0}[\Coh(X)]$,  we have $\cK\in D_{ctf}^b(X,\cE + \cE',\Lambda)$.

\item\label{locality} Let $(f_i : U_i\to X)_{i\in I}$ be an étale cover $X$.
Then $\cK\in D_{ctf}^b(X,\cE,\Lambda)$ if and only if for every $i\in I$, we have $\cK|_{U_i}\in D_{ctf}^b(U_i,f_i^*\cE,\Lambda)$.

\item\label{triangle}  Consider a distinguished triangle
$$
\cK_1\to \cK_2 \to \cK_3\to
$$
in $D_{ctf}^b(X,\Lambda)$.
If 2 out of the 3 above complexes lie in $D_{ctf}^b(X,\cE,\Lambda)$,  so does the third.

\item\label{sequence} If the above triangle is an exact sequence of constructible sheaves, then $\cK_2$ lies in $D_{ctf}^b(X,\cE,\Lambda)$ if and only if so do $\cK_1$ and $\cK_3$.

\item\label{skyscraper_point} If $\cK_1,\cK_2\in D_{ctf}^b(X,\Lambda)$ are isomorphic away from a finite set of closed points of $X$, then $\cK_1$ has log conductors bounded by $\cE$ if and only if so does $\cK_2$.

\item \label{local_bound_to_global}
Let $(U_i)_{i\in I}$ be a finite Zariski cover of $X$ and let $\cE_i\in \bQ_{\geq 0}[\Coh(U_i)]$, $i\in I$.
Then, there exists $\cE'\in \bQ_{\geq 0}[\Coh(X)]$ depending only on the $\cE_i$ such that 
$$
\cK|_{U_i}\in D_{ctf}^b(U_i,\cE_{i},\Lambda) \text{ for every } i\in I  \Rightarrow   \cK\in D_{ctf}^b(X,\cE',\Lambda) \ .
$$

\item\label{reduction} Let $\red : X^{\red}\to X$ be the reduction morphism.
Then,  $\cK$ lies in $D_{ctf}^b(X,\cE,\Lambda)$ if and only if $\red^*\cK$ lies in $D_{ctf}^b(X,\red^*\cE,\Lambda)$.

\item\label{closed_immersion} For every closed immersion $i : X\hookrightarrow Y$,  we have $\cK\in D_{ctf}^b(X,\cE,\Lambda)$ if and only if $i_*\cK\in D_{ctf}^b(Y,i_*\cE,\Lambda)$.

\item\label{devissage_bound}
Let $ j : U\hookrightarrow X$ be an open subset and  $i : Z \hookrightarrow X$ its complement.
Assume that $\cE\in \bQ_{\geq 0}[\Coh(X)]$, that $ j_! (\cK|_U) \in D_{ctf}^b(X,\cE,\Lambda)$ and  $ \cK|_Z  \in D_{ctf}^b(Z,\cE_Z,\Lambda)$ for some $\cE_Z\in \bQ_{\geq 0}[\Coh(Z)]$.
Then,  
$$
\cK\in D_{ctf}^b(X,\cE + i_*\cE_Z,\Lambda) \ .
$$
\end{enumerate}
\end{lem}

\begin{proof}
Items (1)(2)(3)(6) are obvious.
Items (4)(5) follow from \cref{slope_exact}.
For (7), we know by \cite[0G41]{SP} that each $\cE_i\in \bQ_{\geq 0}[\Coh(U_i)]$ extends to an element $\cE'_i\in \bQ_{\geq 0}[\Coh(X)]$.
Then, the sum of the $\cE'_i$ does the job in virtue of  (2),(3).
Item (8) follows from (1) and the fact that a smooth curve mapping to $X$ canonically factors through $X^{\red}$.
The converse implication in (9) follows from (1) and the isomorphism $i^*i_*\cE\simeq \cE$.
For the direct implication, consider the cartesian square of schemes of finite type over $k$
$$
	\begin{tikzcd}
	      X_C  \arrow{r}{f'}   \arrow{d}{i'}      &   X    \arrow{d}{i} \\
	       C            \arrow{r}{f}    & Y  
	\end{tikzcd}
$$
where $C$ is a smooth connected curve over $k$.  
We have 
$$
(i_*\cF)|_C \simeq i'_{\ast}(\cF|_{X_C}) \ .
$$
If  $X_C$ has dimension $0$,  the sheaf $(i_*\cF)|_C$ is generically 0 and there is nothing to prove.
Otherwise,  the composition $X_C^{\red}\to X_C \to C$ is an isomorphism, and the conclusion follows.
Item (10) follows immediately by (2),(4),(9) via the localization triangle.
\end{proof}

The following is our main supply of sheaf with explicit bound on the log conductors.

\begin{prop}\label{bounded_ramification_ex}
Let $X$ be a normal scheme of finite type over $k$.
Let $D$ be an effective Cartier divisor of $X$ and put $j : U:=X-D\hookrightarrow X$.
Let  $\cL\in \Loc_{tf}(U,\Lambda)$ and  $\cE\in \bQ[\Coh(X)]$.
\begin{enumerate}\itemsep=0.2cm
\item\label{bounded_ramification_ex_1} 
If $j_!\cL\in D_{ctf}^b(X,\cE,\Lambda)$, then $LC(j_!\cL) \leq T(\cE)$.
\item \label{bounded_ramification_ex_2} 
If $X$ is smooth, $j_!\cL$ has log conductors bounded by $C(j_!\cL)$.
\item\label{bounded_ramification_ex_3} 
If $X$ is smooth and if $D$ has simple normal crossings, $j_!\cL$ has log conductors bounded by $LC(j_!\cL)$.
\end{enumerate}

\end{prop}
\begin{proof}
Item (1) is a local question around the generic points of $D$.
Hence, we can suppose that $X$ and $D$ are smooth connected.
We argue for $\cE\in Coh(X)$, the general case being similar by linear combinations.
At the cost of shrinking $X$ further, we can suppose that $\cE$ is of the form 
$$
\cO_X^r \oplus \cO_X/\cI_D^{n_1}\oplus \cdots  \oplus  \cO_X/\cI_D^{n_k}
$$
where $n_i \in \bN^*$.
In particular $T(\cE)=(n_1+\cdots + n_k) \cdot D$.
Let $f : C\to X$ be an immersion over $k$ where $C$ is a smooth connected curve over $k$ such  that the generic point of $C$ is sent in $U$ and  that $f^{-1}(D)$ is a single point $x$.
Then, 
$$
f^*\cE \simeq \cO_C^r \oplus \cO_C/\cI_x^{m_x(f^*D)\cdot n_1}\oplus \cdots  \oplus  \cO_C/\cI_x^{m_x(f^*D)\cdot n_k} \ .
$$
In particular, 
$$
\lc_x((j_!\cL)|_C)  \leq  m_x(f^*D) \cdot m_D(T(\cE)) \ .
$$
Then (1) follows from \cref{Hu_generic_swan_via_curve}.
Item (2) follows from  \cref{semi_continuity} and \cref{inequalityLogNonLog_divisor_rem}.
Item (3) follows from \cref{semi_continuity}.
\end{proof}

\begin{prop}\label{existence_bound}
Let $X$ be a scheme of finite type over $k$.
For every $\cK\in D_{ctf}^b(X,\Lambda)$, there exists  $\cE\in \Coh(X)$ such that $\cK$ has log conductors bounded by $\cE$.
\end{prop}

\begin{proof}
If $X$ has dimension 0, there is nothing to do.
Let  $n\geq 1$, and assume that \cref{existence_bound} holds in dimension $\leq n-1$.
Let $X$ be a scheme of finite type over $k$ of dimension $\leq n$ and let $\cK\in D_{ctf}^b(X,\Lambda)$.
To show that \cref{existence_bound} holds for $\cK$, we can suppose by \cref{many_properties}-\ref{triangle} that $\cK$ is concentrated in degree $0$.
By using a finite affine cover of $X$,  we can assume by  \cref{many_properties}-\ref{local_bound_to_global} that $X$ is affine.
By Noether normalization, there is a finite morphism $f : X\to \mathds{A}^m_k$ where $m\leq n$.
In particular, the counit map $f^*f_*\cK\to \cK$ is surjective.
By \cref{many_properties}-\ref{sequence}, we are thus left to prove \cref{existence_bound} for $f^*f_*\cK$. 
By \cref{many_properties}-\ref{pullback}, we are left to prove \cref{existence_bound} for $f_*\cK$.
Hence, we can suppose that $X$ is smooth of dimension $\leq n$.
Let $Z\subset X$ be an effective Cartier divisor containing the singular locus of $\cK$ and put $j : U:=X-Z\hookrightarrow X$.
By recursion assumption,  $\cK|_Z  \in D_{ctf}^b(Z,\cE_Z,\Lambda)$ for some $\cE_Z \in \Coh(Z)$.  
By \cref{many_properties}-\ref{devissage_bound},  we can thus further suppose that $\cK=j_!\cL$ where 
$\cL \in \Loc_{tf}(U,\Lambda)$.
In that case, $j_!\cL$ has log conductors bounded by $C(j_!\cL)$ in virtue of \cref{bounded_ramification_ex}.
\end{proof}

\begin{prop}\label{bounded_ramification_generic_fiber}
Let $f : X\to S$ be a  morphism between schemes of finite type over $k$.
Let $\cE\in \bQ[\Coh(X)]$ and $\cK\in D_{ctf}^b(X,\cE,\Lambda)$.
For every algebraic geometric point $\overline{s}\to S$, the complex $\cK|_{X_{\overline{s}}}$ has log conductors bounded by $i_{\overline{s}}^*\cE$ where $i_{\overline{s}} : X_{\overline{s}} \to X$ is the canonical morphism.

\end{prop}

\begin{proof}
We can suppose that $\cK$ is concentrated in degree 0.
Let $s\in S$ be the point over which $\overline{s}$ is localized and let $S'$ be a smooth connected open subset  of $\overline{\{s\}}\subset S$.
Consider the  commutative diagram with cartesian squares
$$
\xymatrix{
X_{\overline{s}}\ar[d] \ar[r]^{i_{\overline{s}}'}&X'\ar[d]^{f'}      \ar[r]^{h} &X\ar[d]^{f} \\
\overline{s}\ar[r]&S'  \ar[r] & S      \ .
}
$$
By \cref{many_properties}-\ref{pullback},  the sheaf $\cK|_{X'}$ has log conductors bounded by $h^*\cE$.
Hence, at the cost of replacing $f : X\to S$ by $f' : X'\to S'$ and $\cK$ by $\cK|_{X'}$, we can suppose that $S$ is smooth connected and that $\sbar$  is localized at the generic point of $S$.
Let $C$ be a smooth curve over $\sbar$ and let $h : C \to X_{\sbar}$  be a morphism over $\sbar$.
We want to show that 
$$
LC(\cK|_{C}) \leq T(h^* i_{\sbar}^*\cE) \ .
$$
By writing  $\sbar$   as a filtered limit of smooth connected varieties over $k$  quasi-finite flat over $S$,  there is a spreading out
$$
	\begin{tikzcd}
	           C               \arrow{r}{\iota_{\sbar}}    \arrow{d}{h}    &  \cC      \arrow{d} \arrow{rd}{\alpha}       &     \\
	          X_{\sbar}                 \arrow{r} \arrow{d} &   X_T \arrow{d}    \arrow{r}  & X \arrow{d}\\
                 \sbar       	 \arrow{r}{ } &       T        \arrow{r}{ }      & S
	\end{tikzcd}
$$
with cartesian squares where $T\to S$ is quasi-finite flat with $T$ smooth connected over $k$, and where $\cC \to T$ is a smooth relative curve.
Since $\cK|_{\cC}$ has log conductors bounded by $\alpha^*\cE$, it is enough to show that if $f : X\to S$ is a smooth relative curve  with $S$ smooth connected over $k$ and if $\sbar\to S$ is localized at the generic point of $S$, we have
$$
LC(\cK|_{X_{\sbar} }) \leq T(i_{\sbar}^*\cE) \ .
$$
Observe that $X$ is smooth over $k$.
Choose a reduced effective Cartier divisor $D\subset X$ containing the singular locus of $\cK$.
At the cost of shrinking $S$, we can suppose that $D$ is flat over $S$.
Put $U:=X-D$ and let $j : U\hookrightarrow X$ be the inclusion.
Note that $(j_!(\cK|_U))|_{X_{\sbar}}$ and $\cK|_{X_{\sbar}}$ are isomorphic away from a finite number of points of the smooth curve $X_{\sbar}$.
Hence, at the cost of replacing $\cK$ by $j_!(\cK|_U)$, we can suppose that $\cK$ is of the form $j_!\cL$ where $\cL\in \Loc_{tf}(U,\Lambda)$.
By \cite[0550]{SP} there is a commutative diagram with cartesian squares 
$$
	\begin{tikzcd}
	          X_{\sbar}                 \arrow{r} \arrow{d} &   X_T \arrow{d}    \arrow{r}  & X \arrow{d}\\
                \sbar      	 \arrow{r}{ } &       T        \arrow{r}{ }      & S
	\end{tikzcd}
$$
where $T\to S$ is quasi-finite flat with $T$ smooth connected over $k$ and where the generic fibre of $(T\times_S D)^{\red}\to T$ is geometrically reduced.
At the cost of pulling back the situation to $T$ and replacing $D$ by $(T\times_S D)^{\red}$, we can assume that $D_{\sbar}$ is reduced.
Since $D_{\sbar}$ is a finite set of points over $\sbar$, we deduce that $D_{\sbar}$ is étale over $\sbar$.
At the cost of shrinking $S$ we can thus suppose that the induced map $D\to S$ is étale.
Hence, we deduce

$$
 LC((j_!\cL)|_{X_{\sbar}}) \leq i_{\sbar}^*  LC(j_!\cL) \leq  i_{\sbar}^*T(\cE) = T(i_{\sbar}^*\cE) 
$$
where the first inequality follows from \cref{inequality_LC} and the second inequality follows from \cref{bounded_ramification_ex}.
Since $\sbar$ is localized at the generic point of $S$, the map $i_{\sbar}  : X_{\sbar} \to X$ is flat. 
On the other hand, $X$ and  $X_{\sbar}$ are regular. 
Thus the last equality follows from \cref{flat_pullback_reg}.
\end{proof}

\section{Conductor and finite direct image: global case}

\begin{lemma}\label{bound_direct_image_dim_1}
Let $f: Y\to X$ be a finite surjective morphism between smooth curves over $k$.
Let $D$ be a divisor on $X$ and let $x\in D$. 
Put $E=f^{-1}(D)$  and $j : V:=Y-E\hookrightarrow Y$.
Then, for every $\cL\in \Loc_{tf}(V,\Lambda)$, we have 
$$
\lc_x(f_{*} j_!\cL)\leq \max\{\lc_x(f_{*} j_!\Lambda),\lc_y( j_!\cL),y\in f^{-1}(x)\} \ .
$$
\end{lemma}
\begin{proof}
This is a geometric rephrasing of \cref{slope_induction_final}.
\end{proof}

\begin{definition}
Let $f:Y\to X$ be a finite surjective morphism of normal schemes of finite type over $k$. 
Let $D$ be an irreducible divisor on $X$ with generic point $\eta$ and $E$ an irreducible component of $(D\times_XY)^{\red}$ with generic point $\xi$.
We denote by $f(E/D)$ the degree of $k(\xi)/k(\eta)$, by $e(E/D)$ the ramification index of the extension of discrete valuation rings $\mathscr O_{Y,\xi}/\mathscr O_{X,\eta}$,  by $f(E/D)^{\mathrm{s}}$ the separable degree of $k(\xi)/k(\eta)$ and by $f(E/D)^{\mathrm{ins}}$ its purely inseparable degree.
\end{definition}

\begin{lemma}\label{multiplicity bound}
Let $f:Y\to X$ be a finite surjective morphism of smooth schemes over $k$. 
Let $D$ be a smooth irreducible effective Cartier divisor on $X$ and put $U:=X-D$.
Assume that $E:=(D\times_X Y)^{\red}$ is  smooth irreducible and put $V:=Y-E$.
Let $\iota:C\to X$ be an immersion from a smooth connected curve over $k$ meeting $D$ at a unique closed point $x$.
Let $\iota' : C'\to Y$ be the normalization of $C\times_XY$. 
Then, for every closed point $x'\in C'$ lying over $x$, we have 
$$
m_{x'}(\iota'^*E) \leq   (D,C)_x \cdot  f(E/D)  \ .
$$
\end{lemma}
\begin{proof}
At the cost of throwing away the components of $Y$ not containing $E$, we can suppose that $Y$ is connected.
Let $\eta$ be the generic point of $D$ and let $\xi$ be the generic point of $E$.
By Zariski's main theorem,  the normalization of $\mathcal{O}_{X,\eta}$ in $K(Y)$ is $\mathcal{O}_{Y,\xi}$.
In particular,  \cite[Proposition I.10]{CL} gives
$$
[K(Y):K(X)]=f(E/D)\cdot e(E/D) \ .
$$
Let $S \subset C'$ be the irreducible component containing $x'$.
Then, 
$$
e(x'/x) \leq [K(S):K(C)] \leq [K(Y):K(X)] 
$$
where the first inequality comes from \cite[Proposition I.10]{CL}  and the second   from the fact that  $V\to U$ is finite \textit{flat} as finite surjective morphism with smooth base and target.
Consider the following commutative diagram 
$$
\xymatrix{\relax
E\ar[d] \ar[r]&Y\ar[d]_f&C'\ar[d]^{f_C}\ar[l]_-(0.5){\iota'} \\
D\ar[r]&X&\ar[l]^-(0.5){\iota}C     \ .
}
$$
Then, 
\begin{align*}
m_{x'}(\iota'^*E)&=\frac{1}{e(E/D)}\cdot m_{x'}\left(\iota'^*f^*D\right)=\frac{1}{e(E/D)}\cdot m_{x'}\left(f_C^*\iota^*D\right)=\frac{(D,C)_x }{e(E/D)}\cdot m_{x'}\left(f_C^*x\right)\\
&=(D,C)_x \cdot\frac{e(x'/x)}{e(E/D)}\leq  (D,C)_x \cdot \frac{[K(Y):K(X)]}{e(E/D)}=(D,C)_x\cdot f(E/D).
\end{align*}
\end{proof}

\begin{lemma}\label{FDILocal_irreducible_radicial_log}
Let $f:Y\to X$ be a finite surjective morphism of smooth schemes over $k$. 
Let $D$ be a smooth effective Cartier irreducible divisor on $X$.
Assume that $E:=(D\times_X Y)^{\red}$ is smooth irreducible.
Put $U:= X-D$ and $j : V:=Y-E\hookrightarrow Y$.
Assume that the restriction $f_0:V\to U$ is étale. 
For $\cL\in \Loc_{tf}(V,\Lambda)$, we have 
$$
\lc_D(f_{*}j_!\cL)\leq \max\{\lc_D(f_{*}j_!\Lambda), f(E/D)\cdot \lc_{E}(j_!\cL)\}.
$$
\end{lemma}

\begin{proof}
Let $\iota:C\to X$ be an immersion from a smooth curve over $k$ meeting $D$ at only one point $x$.
Put $C_0:=C-\{x\}$ and let $C'$ be the normalization of $C\times_XY$. 
Put  $C'_0:= C' - (x\times_CC')$.
We have the following commutative diagrams
$$
\xymatrix{\relax
E\ar[d]_{f_D}\ar[r]&Y\ar[d]_f&C'\ar[d]^{f_C}\ar[l]_-(0.5){\iota'}& &V\ar[d]_{f_0}&C'_0\ar[d]^{f_{C_0}}\ar[l]_-(0.5){\iota'_0}\\
D\ar[r]&X&\ar[l]^-(0.5){\iota}C& &U&C_0\ar[l]^-(0.5){\iota_0}}
$$
Since $f_0:V\to U$ is étale and since smoothness descends along étale morphisms,  the scheme $V\times_U C_0$ is smooth over $k$.
Hence,  the right square above is cartesian.
By proper base change, we deduce
$$
\iota^*_0 f_{0*}  (\cL)  \simeq f_{C_{0}*}\iota'^*_0(\cL) \ .
$$
Hence, we deduce that
$$
\iota^* f_{*} ( j_!\cL)  \simeq f_{C*}\iota'^*(j_!\cL) \ .
$$   
On the other hand,
\begin{align*}
\lc_{x}(f_{C*}\iota'^*(j_!\cL) ) &\leq \max\left\{ \lc_x(f_{C*}\iota'^*(j_!\Lambda) ), \lc_{x'}(\iota'^*j_!\cL),x'\in f_C^{-1}(x)  \right\}     & \text{\cref{bound_direct_image_dim_1}} \\
  &   \leq   \max\left\{ lc_x(f_{C*}\iota'^*(j_!\Lambda) ), m_{x'}(\iota'^*E)\cdot \lc_{E}(j_!\cL) ,x'\in f_C^{-1}(x)  \right\}           &  \text{\cref{semi_continuity}}    \\
                                           &    \leq   \max\left\{ \lc_x(\iota^* f_{*} (j_!\Lambda) ), (D,C)_x \cdot   f(E/D)\cdot \lc_{E}(j_!\cL)   \right\}     & \text{\cref{multiplicity bound}}
\end{align*}
Thus, 
\begin{align*}
\frac{\lc_{x}(\iota^* f_{*} ( j_!\cL)  )}{(D,C)_x} & \leq \max\left\{ \frac{\lc_x(\iota^* f_{*} (j_!\Lambda) )}{(D,C)_x},      f(E/D)\cdot \lc_{E}(j_!\cL) \right\}    &   \\
     &  \leq \max\left\{ \lc_{D}(f_{*} ( j_!\Lambda ) ),  f(E/D)\cdot \lc_{E}(j_!\cL) \right\}  &    \text{\cref{semi_continuity}}  
\end{align*} 
Applying \cref{Hu_generic_swan_via_curve} concludes the proof of \cref{FDILocal_irreducible_radicial_log}.
\end{proof}

\begin{lemma}\label{FDILocal_irreducible_radicial}
Let $f:Y\to X$ be a finite surjective morphism of smooth schemes over $k$. 
Let $D$ be a smooth irreducible effective Cartier divisor on $X$ and put $U:= X-D$.
Assume that $E:=(D\times_X Y)^{\red}$ is  smooth irreducible and put $j : V:=Y-E\hookrightarrow Y$.
Assume that the restriction $f_0:V\to U$ is étale. 
For every $\cL\in \Loc_{tf}(V,\Lambda)$,  we have 
$$
c_D(f_{*}j_!\cL)\leq \max\{ c_D(f_{*}j_!\Lambda), f(E/D)\cdot \mathrm c_{E}(j_!\cL)\} \ .
$$
\end{lemma}
\begin{proof}
By \cref{equality_DT} applied to $f_{*}j_!\cL$ and $f_{*}j_!\Lambda$,  there is an immersion   $\iota:C\hookrightarrow X$ from a smooth curve over $k$ and a point $x\in C\cap D$ such that 
$$
c_{D}(f_{*}j_!\cL)  =c_{x}(\iota^*f_{*}j_!\cL) \text{ and }
 c_{D}(f_{*}j_!\Lambda) =c_{x}(\iota^*f_{*}j_!\Lambda)    \ .
$$  
Let $C'$ be the normalization of $C\times_XY$. 
We have the following commutative diagrams
$$
\xymatrix{\relax
E\ar[d]_{f_D}\ar[r]&Y\ar[d]_f&C'\ar[d]^{f_C}\ar[l]_-(0.5){\iota'}\\
D\ar[r]&X&\ar[l]^-(0.5){\iota}C}
$$
As in the proof of \cref{FDILocal_irreducible_radicial_log}, we have 
$$
\iota^* f_{*} ( j_!\cL )  \simeq f_{C*}\iota'^*(j_!\cL ) \ .
$$   
Hence, we deduce
$$
c_{D}(f_{*}j_!\cL)  =c_{x}(f_{C*}\iota'^*(j_!\cL ) ) \text{ and }
 c_{D}(f_{*}j_!\Lambda) =c_{x}(f_{C*}\iota'^*(j_!\Lambda ) )    \ .
$$  
On the other hands, we have
\begin{align*}
c_{x}(f_{C*}\iota'^*(j_!\cL) ) &= \lc_{x}(f_{C*}\iota'^*(j_!\cL) )  +1   & \text{\cref{inequalityLogNonLog}} \\
&\leq \max\left\{ \lc_x(f_{C*}\iota'^*(j_!\Lambda )), \lc_{x'}(\iota'^*(j_!\cL) ),x'\in f_C^{-1}(x)  \right\}  +1   & \text{\cref{bound_direct_image_dim_1}} \\
&\leq \max\left\{ c_x(f_{C*}\iota'^*(j_!\Lambda ) ), c_{x'}(\iota'^*(j_!\cL) ),x'\in f_C^{-1}(x)  \right\}   & \text{\cref{inequalityLogNonLog}} \\
  &   \leq   \max\left\{ c_x(f_{C*}\iota'^*(j_!\Lambda )), m_{x'}(\iota'^*E)\cdot c_{E}(j_!\cL) ,x'\in f_C^{-1}(x)  \right\}           &  \text{\cref{semi_continuity}}    \\  
                                           &    \leq   \max\left\{ c_x(f_{C*}\iota'^*(j_!\Lambda )),   f(E/D)\cdot c_{E}(j_!\cL)   \right\}     & \text{\cref{multiplicity bound}}
\end{align*}

\cref{FDILocal_irreducible_radicial} thus follows.

\end{proof}

\begin{prop}\label{FDILocal}
Let $f:Y\to X$ be a finite surjective morphism of normal schemes of finite type over $k$. 
Let $D$ be an irreducible effective Cartier divisor on $X$ and put $U:=X-D$.
Define  $E:=(D\times_X Y)^{\red}$  and put $j : V:=Y-E\hookrightarrow Y$.
Let $\{E_i\}_{1\leq i\leq m}$ be the irreducible components of $E$. 
Assume that the restriction $f_0:V\to U$ is étale. 
For every $\cL\in \Loc_{tf}(V,\Lambda)$, we have 
$$
c_D(f_{*}j_!\cL)\leq \max_{1\leq i\leq m}\{c_D(f_{*}j_!\Lambda), f(E_i/D)^{\mathrm{ins}}\cdot c_{E_i}(j_!\cL)\}
$$
and 
$$
\lc_D(f_{*}j_!\cL)\leq \max_{1\leq i\leq m}\{\lc_D(f_{*}j_!\Lambda), f(E_i/D)^{\mathrm{ins}}\cdot \lc_{E_i}(j_!\cL)\}.
$$
\end{prop}

\begin{proof}
To lighten the notations, we will omit the lower shrieks in what follows.
Let $\eta$ be the generic point of $D$ and $\overline\eta\to X$ an algebraic geometric point above $\eta$. 
Let $\xi_i$ be the generic point of $E_i$ $(1\leq i\leq m)$, and $\overline\xi_i\to Y$ an algebraic geometric point above $\xi_i$ such that the composition $\overline\xi_i\to\xi_i\to\eta$ factors through $\overline\eta\to\eta$.
By pulling-back above the strict henselianization of $X$ at $\overline\eta$ and then spreading out, there exists a commutative diagram
\[
\xymatrix@C=0.8em{
\widetilde V=\displaystyle{\coprod_{1\leq i\leq m}}\coprod_{f(E_i/D)^{\mathrm{s}}} \widetilde V_i \ar[rr]\ar[dr] \ar[dd]_(0.4){\widetilde f_0}  &   & \widetilde Y=\displaystyle{\coprod_{1\leq i\leq m}}\coprod_{f(E_i/D)^{\mathrm{s}}}\widetilde Y_{i}\ar'[d][dd]_(-0.4){\widetilde f}   \ar[rd] &&  \widetilde E=\displaystyle{\coprod_{1\leq i\leq m}}\coprod_{f(E_i/D)^{\mathrm{s}}} \widetilde E_i \ar[dr]\ar'[d][dd] \ar[ll]\\ 
 & V  \ar[rr] \ar[dd]_(0.4){f_0} &  &  Y\ar[dd]_(0.4){f} && E \ar[dd] \ar[ll]  \\
      \tilde{U}  \ar'[r][rr]  \ar[dr]&   & \tilde{X} \ar[dr] &&   \tilde{D}\ar[rd] \ar'[l][ll]\\
   &  U \ar[rr]&   & X &&   D  \ar[ll]
}
\]
where every square but the right front and back squares are cartesian,  and where $\widetilde X$ is an affine smooth connected \'etale neighborhood of $\overline\eta\to X$ such that $\widetilde D$ is irreducible and smooth,  $\widetilde Y_{i}$ are affine smooth connected \'etale neighborhood of $\overline\xi_i\to Y$ with $\widetilde E_i=\widetilde Y_i\times_YE_i$  irreducible and smooth  and $\widetilde E_i\to \widetilde D$  finite, surjective and radiciel of degree $f(E_i/D)^{\mathrm {ins}}$ for $1\leq i\leq m$.
Since $\widetilde X\to X$ is \'etale, the proper base change yields 
\begin{align*}
c_{D}(f_{*}\Lambda) &=  c_{\widetilde D}(\widetilde f_{*}\Lambda)=\max_{1\leq i\leq m}\{c_{\widetilde D}(\widetilde f_{i*}\Lambda)\};\\
 c_{ D}(f_{*}\cL) &=c_{\widetilde D}(\widetilde f_{*}(\cL|_{\widetilde V})) =\max_{1\leq i\leq m}\{c_{\widetilde D}(\widetilde f_{,i*}(\cL|_{\widetilde V_i}))\} ; \\
c_{ E_i}(\cL)&=  c_{\widetilde E_i}(\cL|_{\widetilde V_i});
\end{align*}
and similarly with the logarithmic conductors, where $\widetilde f_{i}:\widetilde Y_i\to \widetilde X$ is induced by $f : Y\to X$.
Hence, it is enough to prove 
$$
c_{\widetilde D}(\widetilde f_{i*}(\cL|_{\widetilde V_i}))\leq \max\{c_{\widetilde D}(\widetilde f_{i*}\Lambda), f(\widetilde E_i/\widetilde D) \cdot c_{\widetilde E_i}(\cL|_{\widetilde V_i})\}
$$
and similarly with the logarithmic conductors.
Thus, we are left to prove \cref{FDILocal} in the case where $D\subset X$ and $E=(D\times_XY)_{\mathrm{red}}\subset Y$ are irreducible and smooth.
This case follows from \cref{FDILocal_irreducible_radicial_log} and \cref{FDILocal_irreducible_radicial}.

\end{proof}

\begin{cor}\label{remark_bound_conductor_direct_image}
In the situation from \cref{FDILocal}, we have 
$$
c_{D}(f_{*}j_!\cL)\leq c_{D}(f_{*}j_!\Lambda) + d\cdot c_{E}(j_!\cL)
$$
and 
$$
\lc_{D}(f_{*}j_!\cL)\leq \lc_{D}(f_{*}j_!\Lambda) +  d\cdot  \lc_{E}(j_!\cL)
$$
where $d$ is the degree of $f : X \to Y$.
Notice that $d$, $c_{D}(f_{*}j_!\Lambda)$  and $\lc_{D}(f_{*}j_!\Lambda)$ only depend on $f:X\to Y$.
\end{cor}

\begin{thm}\label{Ramification_cycle_finite_direct_image}
Let $f:Y\to X$ be a finite  morphism of normal schemes of finite type over $k$. 
Let $D$ be a reduced effective Cartier divisor on $X$ and put $U:=X-D$.
Define $E:=(D\times_X Y)^{\red}$ and put $j : V:=Y-E\hookrightarrow Y$.
Assume that the restriction $f_0:V\to U$ is étale. 
For $\cL\in \Loc_{tf}(V,\Lambda)$,  we have 
$$
C(f_{*}j_!\cL)\leq C(f_{*}j_!\Lambda)+ f_* C(j_!\cL)
$$
and 
$$
LC(f_{*}j_!\cL)\leq LC(f_{*}j_!\Lambda)+ f_* LC(j_!\cL) \ .
$$
\end{thm}
\begin{proof}
We argue for the non logarithmic conductor divisor as the proof is the same for the logarithmic conductor divisor.
We can assume that $X$ and $Y$ are integral.
If $f:Y\to X$ is not surjective, the sheaf $f_{*}j_!\cL$ is supported on a strict closed subset of $X$. 
In that case $C(f_{*}j_!\cL)= 0$ and there is nothing to prove.
We can thus assume that $X$ and $Y$ are integral and that $f:Y\to X$ is surjective.
The question is local on $X$. 
Hence, we may assume that $D$ is irreducible. 
We denote by $E_i$'s $(1\leq i\leq m)$ the irreducible components of $E$. 
By  \cref{FDILocal}, we have
\begin{align*}
 c_D(f_{*}j_!\cL)&\leq \max_{1\leq i\leq m}\{ c_D(f_{*}j_!\Lambda), f(E_i/D)^{\mathrm{ins}}\cdot  c_{E_i}(j_!\cL)\}\\
&\leq  c_D(f_{*}j_!\Lambda)+ \max_{1\leq i\leq m}\{ f(E_i/D)^{\mathrm{ins}}\cdot  c_{E_i}(j_!\cL)\}\\
&\leq  c_D(f_{*}j_!\Lambda)+ \max_{1\leq i\leq m}\{ f(E_i/D)\cdot c_{E_i}(j_!\cL)\}\\
&\leq  c_D(f_{*}j_!\Lambda)+ \sum_{1\leq i\leq m} f(E_i/D)\cdot  c_{E_i}(j_!\cL).
\end{align*}
\cref{Ramification_cycle_finite_direct_image} thus follows.
\end{proof}

\begin{cor}\label{Ramification_cycle_finite_direct_image_smooth_target}
In the situation from \cref{Ramification_cycle_finite_direct_image}, assume that $f : Y \to X$ is surjective with $X$ smooth and $Y$ connected. 
Let $\cE\in \bQ[\Coh(X)]$.
Then, for every $\cL\in \Loc_{tf}(V,\Lambda)$ such that $j_!\cL$ has log conductors bounded by $\cE$, the direct image $f_{\ast}(j_!\cL)$ has log conductors bounded by 
$$
C(f_{*}j_!\Lambda)+ f_* (T(\cE)+E) \ .
$$
If furthermore $D$ has simple normal crossings, $f_{\ast}(j_!\cL)$ has logarithmic conductors bounded by 
$$
LC(f_{*}j_!\Lambda)+ f_* T(\cE) \ .
$$
\end{cor}

\begin{proof}
Since $X$ is smooth, $f_{\ast}(j_!\cL)$ has log conductors bounded by $C(f_{\ast}(j_!\cL))$ in virtue of \cref{bounded_ramification_ex}-(2).
By \cref{Ramification_cycle_finite_direct_image}, we have 
$$
C(f_{*}j_!\cL)\leq C(f_{*}j_!\Lambda)+ f_* C(j_!\cL) \ .
$$
By \cref{inequalityLogNonLog_divisor_rem} and  \cref{bounded_ramification_ex}-(1), we have
$$
C(j_!\cL) \leq LC(j_!\cL) +E \leq T(\cE) + E \ .
$$
This proves the first claim.
The second claim follows similarly via \cref{bounded_ramification_ex}-(3).
\end{proof}

\begin{thm}\label{finite_direct_radicial}
Let $f:Y\to X$ be a finite morphism between normal schemes of finite type over $k$. 
Let $D$ be a reduced effective Cartier divisor on $X$ and put $U\colon =X-D$. 
Define $E\colon= (D\times_X Y)^{\red}$ and put $j : V:= Y-E \hookrightarrow Y$.
Assume that the restriction $f_0:V\to U$ is radicial. 
Then, for $\cL\in\Loc_{tf}(V,\Lambda),$ we have 
$$
C(f_*j_!\cL)\leq f_*C(j_!\cL) 
$$
and 
$$
LC(f_*j_!\cL)\leq f_*LC(j_!\cL)\ .
$$
\end{thm}

\begin{proof}
We can assume that $X$ and $Y$ are integral.
If $f:Y\to X$ is not surjective, the sheaf $f_{*}j_!\cL$ is supported on a strict closed subset of $X$. 
In that case $C(f_{*}j_!\cL)= 0$ and there is nothing to prove.
We can thus assume that $X$ and $Y$ are integral and that $f:Y\to X$ is surjective.
The question  is  local in a neighbourhood of the generic points of $D$.
Hence we can assume that $X$ is smooth over $k$ and that $D$ is integral.
Let $\eta$ be the generic point of $D$.
Since $f:Y\to X$ is finite, for every $y\in Y$, we have $ f(\overline{\{y\}})= \overline{\{f(y)\}}$, so that $\dim \overline{\{y\}}=\dim \overline{\{f(y)\}}$.
Hence, $f^{-1}(\eta)$ is the set of generic points of $E$. 
Since $Y$ is normal, we have $f^{-1}(\eta)\cap Y^{\mathrm{sing}}=\emptyset$. 
After replacing $X$ by $X-f(Y^{\mathrm{sing}})$, we can further assume that $Y$ is also smooth over $k$. 
In this case, $f:Y\to X$ is flat. \\ \indent
Let $\xi_1,\dots,\xi_m$ be the generic points of $E$.
Let $R$ be the henselization of $X$ at $\eta$, let $K$ be the fraction field of $R$, let $S_i$ be the henselization of $Y$ at $\xi_i$ $(1\leq i\leq m)$ and let $L_i$ be the fraction field of $S_i$. 
Note that $R$ and the $S_i$'s are henselian discrete valuation rings. 
Since $f:X\to Y$ is finite and flat, we have 
\begin{align*}
\Spec(R)\times_X Y\cong\coprod_{1\leq i\leq m}\Spec(S_i),
\end{align*}
and $R\to S_i$ are finite and flat $(1\leq i\leq m)$. Hence, we have 
$$
\Spec(K)\times_XY=\Spec(K)\times_UV=\coprod_{1\leq i\leq m}\Spec(L_i).
$$
Since  $f_0:V\to U$ is radicial, we get $m=1$. 
Hence, $E$ has only one irreducible component with generic point $\xi$.
Let $S$  be the henselization of $Y$ at $\xi$ and let $L$ be the fraction field of $S$. 
Since $\Spec(K)\times_XY=\Spec(L)$, we have 
$[K(Y):K(X)]=[L:K]$ and $L/K$ is purely inseparable. 
Let $\kappa_K$ and $\kappa_L$ be the residue fields of $R$ and $S$ respectively.
Since $\kappa_K$ is separably closed, the extension $\kappa_L/\kappa_K$ is purely inseparable.
Let $d$ be its degree and let $e$ be the ramification index of $L/K$.
By \cite[Proposition I.10]{CL} we have
$$
[L:K]=d \cdot e\ .
$$
Let $\overline K$ be an algebraic closure of $K$ containing $L$. 
 Since $L/K$ is purely inseparable, the canonical inclusion $\gamma:G_L\to G_K$ is a bijection. 
 Let $M$ be the $\Lambda$-module with  continuous $G_L$-action associated to $\cL|_{\Spec(L)}$ and $N$ the $\Lambda$-module with  continuous $G_K$-action associated to $(f_*j_!\cL)|_{\Spec(K)}$. 
By \cite[Theorem 1.1]{Hu_purely_inseparable}, we have 
\begin{align*}
G^{d r}_K&\subseteq \gamma(G^{r}_L) \ \ \ \textrm{for}\ \ r\geq 1 \ ,\\
G^{d r}_{K,\log}&\subseteq \gamma(G^{r}_{L,\log}) \ \ \ \textrm{for}\ \ r\geq 0 \ . 
\end{align*}
Thus, we have 
\begin{align*}
c_K(N)  \leq d \cdot c_L(M)\ \ \ \textrm{ and }\ \ \ \lc_K(N) \leq  d \cdot  \lc_L(M) \ .
\end{align*}
Thus we deduce
$$
C(f_*j_!\cL)=c_K(N)\cdot D\leq d \cdot c_L(M)\cdot D=c_L(M)\cdot f_*E=f_*C(j_!\cL)
$$
and 
$$
LC(f_*j_!\cL)=\lc_K(N)\cdot D\leq d\cdot \lc_L(M)\cdot D=\cl_L(M)\cdot f_*E=f_*LC(j_!\cL)\ .
$$

\end{proof}

\begin{lem}\label{generic_factorization}
Let $f : Y\to X$ be a finite surjective morphism between integral schemes of finite type over $k$ where $Y$ is normal.
Then,  $f : Y \to X$   factors through  finite surjective  maps  $g : Y \to T$ and $h : T\to X$ where $T$ is a normal scheme of finite type over $k$ satisfying the following :
\begin{enumerate}\itemsep=0.2cm
\item the map $g : Y \to T$ is radicial.
\item there is a dense open subset $U\subset X$ such that the induced map $h^{-1}(U)\to U$ is étale.
\end{enumerate}

\end{lem}

\begin{proof}
Let $\xi$ be the generic point of $Y$ and let $\eta$ be the generic point of $X$.
Since $f : Y\to X$ is finite surjective, we have $f^{-1}(\eta)=\xi$.
Consider the finite field extension $k(Y)/k(X)$.
There exists a unique intermediate extension 
$$
k(X)\subset k(Y)_{\sep} \subset k(Y)
$$ 
such that $k(Y)_{\sep}/k(X)$ is separable and $k(Y)/k(Y)_{\sep}$ is purely inseparable. 
Let $T$ be the normalization of $X$ in $k(Y)_{\sep}$.
Then,  $f : Y \to X$ factors through a finite surjective map $g : Y \to T$ with $k(Y)/k(T)$ purely inseparable followed by a finite surjective map $h : T\to X$ with $k(T)/k(X)$ separable. 
In particular,  $g : Y \to T$ is radicial since for every $f\in \cO_Y$, there is $n\geq 1$ such that 
$$
f^{p^n}\in k(T)\cap \cO_Y=\cO_T  \  . $$
Hence,  we are left to find a dense open subset $U\subset X$ such that the induced map $h^{-1}(U)\to U$ is étale.
By affine base change, the coherent sheaf $h_*\Omega_{T/X}$ vanishes at $\eta$.
Hence, it vanishes on a dense open neighbourhood $U\subset X$.
Since the formation of $\Omega_{T/X}$ commutes with pull-back, we deduce that $\Omega_{T/X}$ vanishes on $h^{-1}(U)$.
That is the pull-back $h^{-1}(U)\to U$ is unramified.
At the cost of shrinking $U$, we can suppose that $h^{-1}(U)\to U$ is étale.
\end{proof}

The statement of the following theorem was kindly suggested by the referee to serve as a common generalization of \cref{Ramification_cycle_finite_direct_image} and \cref{finite_direct_radicial}.

\begin{thm}\label{finite_direct_image_lcc_case}
Let $f : Y\to X$ be a finite  morphism between normal schemes of finite type over $k$.
Let $D$ be a reduced effective Cartier divisor on $X$ and put $U:=X-D$.
Define  $E:=(D\times_X Y)^{\red}$  and put $j : V:=Y-E\hookrightarrow Y$.
For $\cL\in \Loc_{tf}(V,\Lambda)$,  we have 
$$
C(f_{*}j_!\cL)\leq C(f_{*}j_!\Lambda)+ f_* C(j_!\cL)
$$
and 
$$
LC(f_{*}j_!\cL)\leq LC(f_{*}j_!\Lambda)+ f_* LC(j_!\cL) \ .
$$
\end{thm}
\begin{proof}
We argue for the conductor divisor, the logarithmic case being similar.
This is a local question in a neighbourhood of the generic points of $D$.
Hence we can assume that $X$ is smooth affine over $k$.
In particular, every divisor of $X$ is an effective Cartier divisor.
By \cite[0357]{SP}, the schemes $X$ and  $Y$ are finite disjoint unions of normal  integral schemes.
Hence we can assume that $X$ and  $Y$ are  integral with $X$ smooth affine over $k$.
If $f : Y \to X$ is not surjective, then $f$ factors through a strict closed subset of $X$.
In that case, we have $C(f_{*}j_!\cL)=0$ and there is nothing to prove.
Thus, we can assume that $f : Y \to X$ is surjective.
\\ \indent
Let $U_0 \subset U$ be an  open subset  such that $D_0 := X-U_0$ is a divisor.
We argue that \cref{finite_direct_image_lcc_case} holds for $(f,D,\cL)$ if it holds for $(f,D_0,\cL|_{V_0})$.
Indeed, define  $E_0:=(D_0\times_X Y)^{\red}$  and put $j_0 : V_0:=Y-E_0\hookrightarrow Y$.
Proving the inequality
$$
C(f_{*}j_!\cL)\leq C(f_{*}j_!\Lambda)+ f_* C(j_!\cL)
$$
is a local question at the generic points of $D$.
On the other hand, if $Z\subset D_0$ is the union of the irreducible components of $D_0$ not contributing to $D$, the sheaves $j_!\cL$ and $j_{0 !}\cL|_{V_0}$ coincide on $f^{-1}(X-Z)$.
Since $X-Z$ contains every generic point of $D$, the conclusion follows.
\\ \indent
Hence  at the cost of shrinking $U$,  we can assume by \cref{generic_factorization} the existence of a commutative diagram with pullback squares
$$
	\begin{tikzcd}
		    V\arrow{d}{}     \arrow{r}  & Y \arrow{d}{g}\\
 h^{-1}(U)\arrow{d}{}    \arrow{r}  & T \arrow{d}{h}\\
   U      \arrow{r}{ }      & X
	\end{tikzcd}
$$
featuring finite surjective morphisms where $T$ is normal of finite type over $k$, where $h^{-1}(U)\to U$ is étale and $g : Y\to T$ is radicial.
By \cref{finite_direct_radicial} applied to the morphism of normal schemes $g : Y \to T$ and to $(D\times _X T)^{\red}\subset T$, we know that 
$$
C(g_*j_!\cL)\leq g_*C(j_!\cL)\ .
$$
By \cref{Ramification_cycle_finite_direct_image} applied to the morphism of normal schemes $h : T \to X$ and to  $D\subset X$, we know that 
$$
C(h_{*} g_*j_!\cL)\leq C(f_{*}j_!\Lambda)+ h_* C(g_* j_!\cL) \ .
$$
Putting everything together gives the desired inequality.
\end{proof}

\begin{thm}\label{finite_direct_image}
Let $f : Y\to X$ be a finite morphism between  schemes of finite type over $k$.
Let $\Sigma$ be a stratification of $Y$ and let $\cE\in \bQ[\Coh(Y)]$.
Then there exists $\cE'\in \bQ[\Coh(X)]$ such that for every $\cK\in D_{\Sigma,tf}^b(Y,\cE,\Lambda)$, we have $f_{\ast}\cK \in D^b_{ctf}(X,\cE',\Lambda)$.
\end{thm}

\begin{proof}
We can assume that $\cE\in \Coh(Y)$ and look for $\cE'$ inside $\Coh(X)$.
We argue by recursion on the dimension $n$ of $Y$.
If $n=0$, there is nothing to do.
Assume that $n>0$ and that \cref{finite_direct_image} holds in dimension $<n$.
Since the étale topos is insensitive to nilpotents, we can by \cref{many_properties}-\ref{reduction} assume that $X,Y$ are reduced.
By \cref{many_properties}-\ref{local_bound_to_global}, we can suppose that $X,Y$ are affine.
Let  $\cK\in D_{\Sigma,tf}^b(Y,\cE,\Lambda)$.
Since finite direct images are exact, we can suppose that $\cK$ is concentrated in degree 0.
Let $\nu :  Y^{\nu}  \to Y$ be the normalization map.
Since the unit map  $\cK \to  \nu_{\ast}\nu^{\ast}\cK$ is injective, the induced map 
$$
f_{\ast}\cK \to f_{\ast} \nu_{\ast}\nu^{\ast}\cK
$$
is an injective map of constructible sheaves.
By \cref{many_properties}-\ref{sequence}, at the cost of replacing $Y$ by $Y^{\nu}$ and  $f$ by $f\circ \nu$, we can  suppose that $X,Y$ are affine and that $Y$ is normal.
By Noether normalization lemma,  there is a finite morphism  $\pi : X\to \bA^d_k$ for some $d\geq 0$.
Observe that the counit map 
$$
\pi^*\pi_\ast f_\ast \cK \to f_\ast \cK
$$
is surjective.
By \cref{many_properties}-\ref{sequence}, at the cost of replacing $X$ by $\bA^d_k$ and $f$ by $\pi\circ f$, we can thus suppose that $X$ is affine smooth connected over $k$ and  $Y$ is affine normal.
Since a normal scheme is a disjoint union of its irreducible components,
we can further suppose that $Y$ is irreducible. 
\\ \indent
 Let $Z\subset Y$ be a strict closed subset  containing the strata of $\Sigma$ of dimension $<n$.
 Since $f : Y \to X$ is finite, we have 
$$
\dim f(Z)= \dim Z < n \leq \dim X \ .
$$
Thus, we can choose a reduced effective Cartier divisor $D\subset X$ containing $f(Z)$.
Put $D:=X-U$ and consider the following commutative diagram
$$
	\begin{tikzcd}
	    V\arrow{r}{j}   \arrow{d}{}      &   Y    \arrow{d}{f}  & \arrow{l}{i}   \arrow{d}{f_D}   E  \\
	       U           \arrow{r}    & X     & \arrow{l}  D 
	\end{tikzcd}
$$ 
with pullback squares.
Note that $\dim E=n-1<n$.
By proper base change and recursion applied to the finite morphism $f_D : E\to D$, to  $\Sigma_E:=\Sigma \cap E$ and $i^*\cE\in \Coh(E)$, there is $\cE_D\in \Coh(D)$ such that for every  $\cK\in \Cons_{\Sigma, tf}(Y,\cE,\Lambda)$, we have 
$$
(f_{\ast}\cK)|_D \simeq f_{D\ast}(\cK|_E)\in \Cons_{tf}(D,\cE_D,\Lambda) \ .
$$
By \cref{many_properties}-\ref{devissage_bound}, we can thus suppose that $\cK=j_! \cL$ where $\cL \in \Loc_{tf}(V,\Lambda)$.
In that case, we have
$$
C(f_*j_!\cL)\leq C(f_{*}j_!\Lambda) +f_* C(j_!\cL) \leq C(f_{*}j_!\Lambda) +f_*T(\cE) + f_* E:=C(f,\cE)
$$
where the first inequality is supplied by  \cref{finite_direct_image_lcc_case} and where the second inequality follows from \cref{inequalityLogNonLog_divisor_rem} and \cref{bounded_ramification_ex}-\ref{bounded_ramification_ex_1}.
By \cref{bounded_ramification_ex}-\ref{bounded_ramification_ex_2}, we deduce that $f_*j_!\cL$ has log conductors bounded by  $C(f,\cE) $, and the conclusion follows.
\end{proof}

\bibliographystyle{amsalpha} 
\bibliography{biblio}

\end{document}